\documentclass{amsart}

\usepackage{amsmath, amsfonts, amssymb, amsthm, color, dsfont}
\usepackage[utf8]{inputenc}
\usepackage[colorlinks,linkcolor=blue,citecolor=blue]{hyperref}
\usepackage[all,cmtip]{xy}
\usepackage{tikz}
\usetikzlibrary{cd}
\usepackage[normalem]{ulem}

\newcommand{\fs}{\mathfrak{s}}
\newcommand{\ft}{\mathfrak{t}}
\newcommand{\fa}{\mathfrak{a}}

\newcommand{\BB}{\mathcal{A}^{(0)}}
\newcommand{\BC}{\mathbb{C}}

\newcommand{\BQ}{\mathbb{Q}}
\newcommand{\BZ}{\mathbb{Z}}
\newcommand{\BF}{\mathbb{F}}
\renewcommand{\AA}{\mathcal{A}}
\newcommand{\CC}{{\mathcal{C}}}
\newcommand{\DD}{\mathcal{D}}

\newcommand{\EE}{\mathcal{E}}
\newcommand{\Zn}[1]{\BZ/{#1}\BZ}

\newcommand{\s}{\sigma}
\newcommand{\hs}{{\hat{\sigma}}}

\newcommand{\hm}{{\hat{\mu}}}

\newcommand{\hld}{\hat{\lambda}}

\newcommand{\e}{\varepsilon}

\newcommand{\smatp}[1]{s^{(#1)}}

\newcommand{\tH}{{\tilde{H}_{\CC}}}
\newcommand{\ev}{\operatorname{ev}}
\newcommand{\coev}{\operatorname{coev}}

\newcommand{\GL}{\operatorname{GL}}
\newcommand{\SL}{\operatorname{SL}}

\newcommand{\id}{\operatorname{id}}

\newcommand{\End}{\operatorname{End}}
\newcommand{\gal}{\operatorname{Gal}}
\newcommand{\irr}{\operatorname{Irr}}
\newcommand{\FPdim}{\operatorname{FPdim}}
\newcommand{\tr}{\operatorname{tr}}

\newcommand{\ord}{\operatorname{ord}}
\newcommand{\SO}{\operatorname{SO}}

\newcommand{\stab}{\operatorname{Stab}}
\renewcommand{\bot}{\boxtimes}
\newcommand{\Rep}{\operatorname{Rep}}

\newcommand{\lcm}{\operatorname{lcm}}
\newcommand{\PGL}{\operatorname{PGL}}

\renewcommand{\1}{{\mathds{1}}}

\renewcommand{\Vec}{\operatorname{Vec}}

\newcommand{\pt}{{\operatorname{pt}}}
\newcommand{\ad}{{\operatorname{ad}}}
\renewcommand{\sl}{{\mathfrak{sl}}}

\newcommand{\spec}{\operatorname{Spec}}

\newtheorem{thm}{Theorem}[section]
\newtheorem{lem}[thm]{Lemma}
\newtheorem{prop}[thm]{Proposition}
\newtheorem{cor}[thm]{Corollary}
\newtheorem{rmk}[thm]{Remark}

\newtheorem{conj}[thm]{Conjecture}

\theoremstyle{definition}
\newtheorem{defn}[thm]{Definition}
\newtheorem{example}[thm]{Example}

\def\GQ{\gal(\bar\BQ/\BQ)}

\def\ON{{\Omega^N_2}}
\def\O2{{\Omega^n_2}}
\def\res{\operatorname{res}}
\def\Sym{\operatorname{Sym}}

\def\Orb{\operatorname{Orb}}
\def\hta{{\hat{\tau}}}
\def\a{{\alpha}}
\def\o{{\otimes}}
\def\ol{\overline}
\def\B{\mathcal{B}}
\def\htau{\hat{\tau}}
\def\SLZ{\SL_2(\BZ)}
\def\w{\omega}
\newcommand{\jacobi}[2]{{\genfrac(){}{0}{#1}{#2}}}
\newcommand{\UZ}[1]{{(\BZ/#1\BZ)^\times}}
\renewcommand{\spec}[1]{{\operatorname{spec}(#1)}}
\newtheorem*{CPM}{Theorem I}
\newtheorem*{CPMII}{Theorem II}
\def\svec{{\operatorname{sVec}}}

\def\shat{\hat{S}}
\usepackage[normalem]{ulem}
\usepackage{relsize, stackengine, graphicx}
\stackMath
\newcommand{\btsvec}{\!
\mathbin{\mathop{\bot}\limits_{\scalebox{0.6}{$\operatorname{sVec}$}}}\!
}

\newcommand{\ul}[1]{\underline{#1}}
\title{Modular categories with transitive Galois actions}
\author{Siu-Hung Ng}
\address{Department of Mathematics\\
    Louisiana State University\\
    Baton Rouge, LA 70803\\
    U.S.A}
\thanks{The first author was partially supported by the NSF grant DMS-1664418.}
\email{rng@math.lsu.edu}
\author{Yilong Wang}
\address{Department of Mathematics\\
    Louisiana State University\\
    Baton Rouge, LA 70803\\
    U.S.A}
\email{yilongwang@lsu.edu}
\author{Qing Zhang}
\address{Department of Mathematics\\Purdue University\\
West Lafayette\\
IN 47907\\U.S.A.}
\email{zhan4169@purdue.edu}
\date{}

\begin{document}

\maketitle
\begin{abstract}
In this paper, we study modular categories whose Galois group actions on their simple objects are transitive. We show that such modular categories admit unique factorization into prime transitive factors. The representations of $\SLZ$ associated with transitive modular categories are proven to be minimal and irreducible. Together with the Verlinde formula, we characterize prime transitive modular categories as the Galois conjugates of the adjoint subcategory of the quantum group modular category $\mathcal{C}(\mathfrak{sl}_2,p-2)$ for some prime $p > 3$. As a consequence, we completely classify transitive modular categories. Transitivity of super-modular categories can be similarly defined. A unique factorization of any transitive super-modular category into s-simple transitive factors is obtained, and the split transitive super-modular categories are completely classified.
\end{abstract}
%\tableofcontents

\section{Introduction}\label{sec:intro}
Modular categories are spherical braided fusion categories over $\BC$ whose braidings are nondegenerate. The notion of modular categories has evolved from the studies of rational conformal field theory \cite{MS}, topological quantum field theory \cite{TuBook} and the quantum invariants of knots and 3-manifolds such as the Jones polynomial \cite{Jones87, RT}. Moreover, unitary modular categories are the mathematical foundations of topological phases of matter \cite{Wen92} and topological quantum computing \cite{ZhenghanBook, RW18}.  Similar to the role of groups in the study of symmetries, modular categories are natural algebraic objects to  organize ``quantum symmetries''. 

An important family of examples of modular categories is obtained from the quantum group construction \cite{BakalovKirillov, Row05}.  In general, for any simple Lie algebra $\mathfrak{g}$ and a \emph{suitable} root of unity $q \in \BC$, one can construct a modular category by taking the semisimplification of the category of tilting modules of the quantum group $U_{q}(\mathfrak{g})$ specialized at the root of unity $q$ \cite{And92, AP}. The associated 3-manifold invariants \cite{BHMV1, TurWen} and mapping class group representations \cite{BHMV2, FKW} are also well-studied in the literature. 

Modular categories have many striking arithmetic properties, such as the Verlinde formula, which are encoded in the matrices $S$ and $T$ (see Section \ref{sec:pre}). More precisely, let $ \mathfrak{s}:= {\scriptsize \begin{pmatrix} 0 & -1\\1 &0\end{pmatrix}}$ and  $ \mathfrak{t}:= {\scriptsize \begin{pmatrix} 1 & 1\\0 &1\end{pmatrix}}$ be the generators of the modular group $\SL_2(\BZ)$. For any modular category $\mathcal{C}$, the assignment $ \ol \rho_{\CC}: \mathfrak{s}\mapsto S $, $\mathfrak{t}\mapsto T$ defines a projective representation of $\SL_2(\mathbb{Z})$ \cite{BakalovKirillov, TuBook}. Another notable arithmetic property of $\CC$ is the fact that the kernel of $\bar\rho_{\CC}$ is a congruence subgroup whose level is equal to the order of the T-matrix \cite{NS10}. Moreover, $\bar\rho_{\CC}$ admits liftings to linear representations of $\SL_2(\BZ)$ which are also shown to have congruence kernels in \cite{DongLinNg}. In addition, these liftings enjoy certain symmetries under the action of the absolute Galois group $\GQ$. These properties of the liftings are essential to our proofs in this paper.

Since the irreducible characters of the fusion ring of a modular category $\CC$ can be indexed by the set $\irr(\CC)$ of isomorphism classes of simple objects of $\CC$ \cite{dBG, CosteGannon}, the action of $\GQ$ on  these characters induces a permutation  action on $\irr(\CC)$. The number of Galois orbits are also invariants of modular categories.

It is always important to classify mathematical structures of a certain property in any mathematical theory. There has been literature on classifying modular categories by rank \cite{BNRWClassificationByRank,SRW, BGNPRW}, Frobenius-Perron dimension \cite{bruillard2014classification,BPR19} and Frobenius-Schur exponent \cite{BR12,FSexp2}. Note that there are finitely many modular categories up to equivalence for any given rank \cite{RankFinite}. 
The number of Galois orbits plays prominent roles in most of these papers (see also \cite{Creamer, Green}), which leads to the idea of classifying modular categories by the number of Galois orbits.

In this paper, we investigate modular categories with only one Galois orbit, which are called \emph{transitive modular categories}. The smallest nontrivial example of a transitive modular category is the Fibonacci modular category, which can be described by the adjoint subcategory  $\CC(\sl_2, 3)^{(0)}$ of the quantum group category $\CC(\sl_2, 3)$ 
associated to $\sl_2$ at level 3. More generally, the adjoint subcategory $\CC(\sl_2, p-2)^{(0)}$ of $\CC(\sl_2, p-2)$ and its Galois conjugates are prime and transitive modular categories for any prime $p > 3$ (see Proposition \ref{p:B-trans-prime}). Remarkably, up to equivalence, these are all the nontrivial prime transitive modular categories. Moreover, every transitive modular category can be uniquely factorized (up to permutations of factors) into a Deligne product of prime transitive ones. Specifically, we prove the following two major theorems of this paper (cf.~Theorem \ref{thm:prime-transitive} and Theorem \ref{thm:main}).

\begin{CPM}
Let $\CC$ be a nontrivial modular category. Then $\CC$ is prime and transitive if and only if 
$\ord(T)$ is a prime number $p > 3$
and $\CC$ is equivalent to a Galois conjugate of $\CC(\sl_2, p-2)^{(0)}$ as modular categories. 
\end{CPM}

\begin{CPMII}
Let $\CC$ be a nontrivial modular category. Then $\CC$ is transitive if and only if $\CC$ is equivalent to a Deligne product of prime transitive modular categories whose T-matrices have distinct orders. In particular, $\ord(T)$ is a square-free odd integer whose prime factors are greater than 3. 
\end{CPMII}

To prove these theorems, we first study factorizations of transitive modular categories in Section \ref{sec:sec3}. 
For any transitive modular category $\CC$, we denote by $\BQ(S)$ the $\BQ$-extension by adjoining all the entries of  the S-matrix,  and denote by $G_\CC$ the corresponding Galois group over $\BQ$. Our first observation is that the action of $G_\CC$ on $\irr(\mathcal{C})$ is fixed-point free (Proposition \ref{p:sizeofGalois}), and so $\irr(\CC)$ is a $G_\CC$-torsor. Moreover, every fusion subcategory of a transitive modular category is also transitive and modular (Corollary \ref{cor:fus-subcat}). We conclude that any transitive modular category has a unique factorization into a Deligne product of prime transitive modular categories in Theorem \ref{thm:prime-decomp}. In Section \ref{sec:example}, we study the Galois conjugates of modular categories $\CC(\sl_2, k)^{(0)}$ 
at odd level $k$. We show that for any prime $p\ge 5$, every Galois conjugate of $\CC(\sl_2, p-2)^{(0)}$ is prime and transitive. 

Inspired by the Galois symmetries of the representations of $\SLZ$ associated with modular categories, we define the notion of \emph{minimal representations} of $\SL_2(\BZ)$  and the \emph{characteristic 2-group} of a modular category in Section \ref{subsec:mod-group-rep}. The minimal representations of $\SLZ$ associated with a modular category $\CC$ are completely determined by the eigenvalues of the images of $\ft$ (Lemma \ref{l:unique_minimal}). Moreover, the characteristic 2-group of $\CC$ naturally gives rise to a decomposition of any representation of $\SL_2(\BZ)$ associated with $\CC$ (see Proposition \ref{p:HC-decomp}).
By studying these two notions, we prove in Theorem \ref{thm:tran_irr} that any  representation of $\SL_2(\BZ)$ associated with a transitive modular category $\CC$ is minimal and irreducible, and  that the  order of the T-matrix of $\CC$ is odd and square-free. 

We completely classify transitive modular categories in Section \ref{sec:class} by characterizing the prime and transitive modular categories. Using the minimality and the irreducibility of the representations of $\SLZ$ associated with transitive modular categories, we show that the order of the T-matrix of any prime transitive modular category $\CC$ is a prime $p \ge  5$, and it has the same fusion rules as $\CC(\sl_2, p-2)^{(0)}$. Applying the classification result of \cite{FK}, we show that $\CC$ must be a Galois conjugate of $\CC(\sl_2, p-2)^{(0)}$ 
(see 
Theorem \ref{thm:prime-transitive}). Combining with the unique factorization theorem,  the full classification of transitive modular categories is established in Theorem \ref{thm:main}. 

Finally,  we discuss transitive super-modular categories in Section \ref{sec:super}. We classify all the transitive \emph{split} super-modular categories by using  the classification of transitive modular categories (Theorem \ref{thm:split}). Moreover, a unique factorization of transitive super-modular categories into s-simple transitive factors is obtained in Theorem \ref{thm:sup-decomp}. Then we exhibit a family of non-split transitive prime categories over $\svec$ and conjecture that these are all the s-simple transitive super-modular categories up to Galois conjugate.

The paper is organized as follows. In Section \ref{sec:pre}, we
set up notations and
give a brief review on modular categories. 
In Section \ref{sec:sec3}, we define transitive modular categories and derive some fundamental properties of them. In particular, we establish the prime factorization theorem in Theorem \ref{thm:prime-decomp}. In Section \ref{sec:example}, we discuss the prime and transitive modular categories obtained from the quantum group categories $\CC(\sl_2, p-2)$ for any odd prime $p$. In Section \ref{subsec:mod-group-rep}, we study the modular group representations associated with modular categories. We show in Theorem \ref{thm:tran_irr} that the representations associated with transitive modular categories are irreducible and minimal. In Section \ref{sec:class}, we characterize the prime transitive modular categories in Theorem \ref{thm:prime-transitive}, which implies the complete classification of transitive modular categories in Theorem \ref{thm:main}. Finally, in Section \ref{sec:super}, transitive super-modular categories are introduced and studied. We classify split transitive super-modular categories in Theorem \ref{thm:split} and prove a unique factorization theorem of transitive super-modular categories in Theorem \ref{thm:sup-decomp}.

Throughout this paper, we tacitly use the following notations: $\zeta_{n} = \exp(2\pi i/n)$, $\BQ_n = \BQ(\zeta_n)$, and $i = \zeta_4 = \sqrt{-1}$. A subcategory of any category is assumed to a full subcategory, unless stated otherwise.

\section{Preliminaries}\label{sec:pre}
In this section, we recall some basic definitions and notations. The readers are referred to \cite{BakalovKirillov, EGNO, Kassel} for more details.
\subsection{Braided fusion categories}
A \emph{fusion category} is a semisimple, $\BC$-linear abelian, rigid monoidal category with finite-dimensional Hom-spaces and finitely many isomorphism classes of simple objects including the tensor unit $\1$. For any fusion category $\CC$, we denote by $\irr(\CC)$ the set of isomorphism classes of simple objects of $\CC$. When it is clear from the context, we will  denote  the isomorphism class of an object $X$ of $\CC$ by the same notation $X$. 

The Grothendieck group of $\CC$, denoted by $K_0(\CC)$, admits a ring structure given by the tensor product. More precisely, we have $X\otimes Y = \sum_{Z\in\irr(\CC)} N_{X,Y}^{Z} Z$ for any $X, Y\in\irr(\CC)$, where
\begin{equation}\label{eq:fusion-rules}
N_{X,Y}^{Z} := \dim_{\BC}\CC(X\otimes Y, Z)
\end{equation}
are called the \emph{fusion coefficients}. 
The collection of fusion coefficients $N_{X,Y}^Z$ for all $X, Y, Z \in \irr(\CC)$ is referred to as the \emph{fusion rules} of $\CC$.  The \emph{fusion matrix} $N_X$ of $X \in \irr(\CC)$ is defined as $(N_X)_{Z,Y}:=N_{X,Y}^Z$ for any $Y,Z \in \irr(\CC)$. The largest real eigenvalue of  $N_X$, denoted by $\FPdim(X)$, is called the \emph{Frobenius-Perron dimension of $X$}. The \emph{Frobenius-Perron dimension} of $\CC$ is defined as 
$$
\FPdim(\CC) :=\sum_{X \in \irr(\CC)} \FPdim(X)^2\,.
$$  

Let $\CC$ be a fusion category. For any object $X \in \CC$, the left dual of $X$ is a triple $(X^*, \ev_X, \coev_X)$, where $X^*$ is an object of $\CC$, 
$\ev_{X}: X^* \otimes X \to \1$ and $\coev_{X}: \1 \to X \otimes X^{*}$ are respectively the evaluation and coevaluation morphisms associated with the left dual object $X^*$ of $X$. A simple object $X \in \irr(\CC)$ is called \emph{invertible} if $X \otimes X^* \cong \1$. The \emph{pointed} subcategory of $\CC$, denoted by $\CC_\pt$, is the full abelian subcategory generated by the invertible objects of $\CC$. A fusion category $\CC$ is called pointed if $\CC_{\pt} = \CC$. The \emph{adjoint subcategory} of $\CC$, denoted by $\CC_\ad$ or $\CC^{(0)}$, is the full abelian subcategory generated by  the subobjects of $X \o X^*$ for any $X \in \CC$ (cf.~\cite{ENO,GN}). Both $\CC_\pt$ and $\CC_\ad$ are fusion subcategories of $\CC$.

The left duality of $\CC$ can be extended to a contravariant monoidal functor $(-)^*$, and so $(-)^{**}$ defines a monoidal functor on $\CC$. A \emph{pivotal structure} on a fusion category $\CC$ is an isomorphism of monoidal functors $j: \id_\CC \xrightarrow{\cong} (-)^{**}$. A fusion category equipped with a pivotal structure is called a \emph{pivotal fusion category}. If $\CC$ is a pivotal fusion category, then for any $X \in \CC$ and $f \in \End_{\CC}(X)$, the (left) \emph{quantum trace} of $f$ can be defined as
$$
\tr_j(f) := \ev_{X^{*}} \circ ((j_X\circ f) \otimes \id_{X^{*}}) \circ \coev_X \in \End_{\CC}(\1) \cong \BC\,.
$$
A pivotal structure $j$ on $\CC$ is called \emph{spherical} if $\tr_j(f) = \tr_j(f^*)$ for any endomorphism $f$ of $\CC$. A \emph{spherical fusion category} is a fusion category equipped with a   spherical pivotal structure.  When the pivotal structure is clear from the context, we will drop the  subscript $j$. In a pivotal category $\CC$, the \emph{quantum dimension} $d_X$ of any object $X \in \CC$ is defined to be $d_X := \tr(\id_X)$. It has been shown in \cite{ENO} that if $\CC$ is a spherical fusion category, then $d_X$ is a totally real algebraic integer for any object $X \in \CC$. 

The \emph{global dimension} of any fusion category was introduced in \cite[Def.~2.5]{MuI}. If $\CC$ is a spherical fusion category, its global dimension is given by
$$
\dim(\CC) = \sum_{X \in \irr(\CC)}d_X^2\,.
$$
In particular, $\dim(\CC)$ is a totally positive algebraic integer. We will denote the positive square root of $\dim(\CC)$ by $\sqrt{\dim(\CC)}$. 

A \emph{braiding} on a fusion category $\CC$ is a natural isomorphism 
$$\beta_{X, Y}: X \otimes Y \xrightarrow{\cong} Y \otimes X$$ 
satisfying the Hexagon axioms. A fusion category equipped with a braiding is called a \emph{braided fusion category}. 
Let $\CC$ be a braided fusion category, and $\DD \subset \CC$ a collection of objects of $\CC$. The \emph{M\"uger centralizer} of $\DD$ in $\CC$ (cf.~\cite{MugerSF2}), denoted by $C_{\CC}(\DD)$, is the full subcategory of $\CC$ with the collection of objects given by
$$
\{X \in \CC \mid \beta_{Y, X} \circ \beta_{X, Y} = \id_{X \otimes Y},\ \forall\, Y \in \DD\}\,.
$$
It follows directly from the definition of a braiding that $C_{\CC}(\DD)$ is a fusion subcategory of $\CC$. In particular, the fusion subcategory $C_{\CC}(\CC)$ is called the \emph{M\"uger center} of $\CC$, and is denoted by $\CC'$. A braided fusion category $\CC$ (or its braiding $\beta$) is called \emph{nondegenerate} if $\CC'$ is equivalent to $\Vec$, the category  of finite-dimensional vector spaces over $\BC$. A braided fusion category is called a \emph{symmetric fusion category} if $\CC' = \CC$. By Deligne's theorems \cite{Del1, Del2}, if $\CC$ is a symmetric fusion category, then $\dim(\CC) \in \BZ$. 

\subsection{Modular categories and arithmetic invariants}\label{subsec:pre}

A \emph{premodular} category (or a ribbon fusion category) is a spherical braided fusion category. A \emph{modular category} $\CC$ is a premodular category whose underlying braiding $\beta$ is nondegenerate. The (unnormalized) S-matrix of a premodular category $\CC$ is defined to be
$$
S_{X,Y} := \tr(\beta_{Y, X^*}\circ \beta_{X^*, Y}),
\  X, Y \in \irr(\CC)\,.
$$
In particular, $S_{X, \1} = S_{\1, X} = d_X$. It has been proved in \cite{MugerSF2} that a premodular category is modular if and only if its S-matrix is invertible. Moreover, when $\CC$ is modular, the fusion coefficients can be expressed in terms of  the S-matrix by the Verlinde formula (see, for example, \cite{BakalovKirillov}):
\begin{equation}
N_{X, Y}^{Z} =\frac{1}{\dim(\CC)}\sum_{W \in \irr(\CC)} \frac{S_{X, W}S_{Y, W}S_{Z^{*}, W}}{S_{\1, W}}\,.    
\end{equation}

Let $\CC$ be a modular category. A natural isomorphism $\theta: \id_{\CC} \xrightarrow{\cong} \id_{\CC}$, called the \emph{ribbon structure} of $\CC$, can be defined using the spherical pivotal structure of $\CC$ and the Drinfeld isomorphism (cf.~\cite[Sec.~2]{NS07}). The ribbon structure is compatible with the braiding and the duality in the following sense: 
\begin{equation}\label{eq:tw-def}
\theta_{X\otimes Y} = (\theta_{X} \otimes \theta_Y)\circ \beta_{Y, X} \circ \beta_{X, Y} \quad \text{and}\quad \theta_{X^*} = (\theta_X)^*
\end{equation}
for any objects $X, Y \in \CC$. If $X \in \irr(\CC)$, then $\theta_X$ is a nonzero scalar multiple of $\id_X$. We will use the abuse notation to denote both this scalar and the isomorphism itself by $\theta_X$ whenever $X$ is simple. The  T-matrix of $\CC$ is defined to be the diagonal matrix
$$
T_{X,Y} := \delta_{X, Y} \theta_{X},\ X, Y \in \irr(\CC)\,.
$$
It follows from  \cite{Vafa} (see also \cite[Thm.~3.1.19]{BakalovKirillov}) that $ \theta_X$  has finite order for any $X \in \irr(\CC)$, and so does the T-matrix. The pair of matrices ($S$, $T$) is called the \emph{(unnormalized) modular data} of $\CC$. We may denote the modular data of a modular category $\CC$ by ($S_\CC$, $T_\CC$) when the context needs to be clarified.

For any $m \in \BZ$, the $m$-th \emph{Gauss sum} \cite{NSW} of a modular category $\CC$ is defined as
$$
\tau_m(\CC) := \sum_{X \in \irr(\CC)} d_X^2\theta_X^m\,.
$$
If $\gcd(m, \ord(T))=1$, the $m$-th (multiplicative) \emph{central charge} and the $m$-th \emph{anomaly} of $\CC$ are defined as
\begin{equation}\label{eq:anomaly}
\xi_m(\CC) := \frac{\tau_m(\CC)}{|\tau_m(\CC)|}
\quad\text{and}\quad
\alpha_m(\CC) := \xi_m(\CC)^2\,.
\end{equation}
It is well-known that  $|\tau_1(\CC)| = \sqrt{\dim(\CC)}$, and $\xi_m(\CC)$ is a root of unity (cf.~\cite{BakalovKirillov, MugerSF2,  NSW}).

\subsection{Galois actions on modular categories}  \label{s:galois}
Let $\CC$ be a modular category with modular data ($S$, $T$). For any complex matrix $M$, we denote by $\BQ(M)$ the field extension of $\BQ$ by adjoining the entries of $M$. It has been proved in \cite{NS10} that  $\BQ(S) \subset \BQ(T)=\BQ_N$, where $N =\ord(T)$. In particular, $\BQ(S)$ is an abelian extension over $\BQ$, and its Galois group is denoted by $G_\CC$. It is immediate to see that 
$$
\BQ(S) = \BQ(S_{X,Y}/d_Y\mid X,Y \in \irr(\CC)).
$$

By the Verlinde formula, for any $Y \in \irr(\CC)$, the assignment $$\chi_Y: \irr(\CC) \to \BC,\  X \mapsto \frac{S_{X, Y}}{d_Y}$$ 
defines a character of the fusion ring $K_0(\CC)$, and $\{\chi_Y\mid Y \in \irr(\CC)\}$ is the set of irreducible characters of $K_0(\CC)$ (cf.~\cite{BakalovKirillov}).  Thus, for any $\s \in G_\CC$, $\s(\chi_Y) = \chi_{\hs(Y)}$  for some permutation $\hs$ on $\irr(\CC)$, and the map 
$$G_\CC \to \Sym(\irr(\CC)), \ \s \mapsto \hs$$ 
is a group monomorphism. The set of orbits under this $G_\CC$-action is abbreviated as $\Orb(\CC)$. We will denote a Galois automorphism $\s \in G_\CC$ as well as its associated permutation on $\irr(\CC)$ by $\hs$.

For any Galois extension $E$ over $\BQ(S)$, the Galois group $\gal(E/\BQ)$ acts on $\irr(\CC)$ via the restriction on $\BQ(S)$ or the surjection $\gal(E/\BQ) \xrightarrow{\res} G_\CC$. Therefore, the $\gal(E/\BQ)$-orbits in $\irr(\CC)$ are identical to the $G_\CC$-orbits. 
Using the above convention, for any $\s \in \gal(E/\BQ)$, we use $\hs_{\CC}$ to represent the restriction of $\s$ on $\BQ(S)$ and also its permutation on $\irr(\CC)$. When it is clear from the context, $\hs_{\CC}$ will simply be denoted by $\hs$. In particular, one can take $E=\bar \BQ$ and so the absolute Galois group $\GQ$ acts on $\irr(\CC)$. According to \cite{dBG}, for any $\s \in \GQ$, $X, Y \in \irr(\CC)$, we have
\begin{equation}\label{eq:sgn-1}
\s\left(\frac{S_{X, Y}}{\sqrt{\dim(\CC)}}\right) = \pm \frac{S_{\hs(X), Y}}{\sqrt{\dim(\CC)}}\,.
\end{equation}

If $\CC = \AA\bot \mathcal{B}$ for some modular categories $\AA$ and $\mathcal{B}$, then $\irr(\CC) = \irr(\AA) \times \irr(\mathcal{B})$ under the identification $X \bot Y \mapsto (X, Y)$. In this case, $S_{\CC} = S_{\AA} \otimes S_{\mathcal{B}}$,  the Kronecker product of the S-matrices. Therefore, $\BQ(S_\CC)=\BQ(S_\AA)\BQ(S_\B)$, the composite field of $\BQ(S_\AA)$ and $\BQ(S_\B)$. Let $\BF = \BQ(S_\AA) \cap \BQ(S_\B)$, and $L=\gal(\BF/\BQ)$.  The restrictions on $\BF$ define two epimorphisms $\res_\AA: G_\AA \to L$ and $\res_\B: G_\B \to L$. Their fiber product $G_\AA \bullet G_\B$, defined as
$$
G_\AA \bullet G_\B :=\{(\hs, \htau) \in G_\AA \times G_\B\mid \res_\AA(\hs)= \res_\B(\htau)\},
$$
satisfies the commutative diagram
$$
\xymatrix{
G_\AA \bullet G_\B \ar[r]^-{p_\AA} \ar[d]_{p_\B} & G_\AA \ar[d]^-{\res_\AA} \\
G_\B \ar[r]^-{\res_\B} & L
}
$$
where $p_\AA,  p_\B$ are coordinate projections.
By the universal property of the fiber product, the restriction epimorphisms $\pi_\AA :G_\CC \to G_\AA$ and $\pi_\B: G_\CC \to G_\B$ induce a group homomorphism 
\begin{equation}\label{eq:gal-iso}
f: G_\CC \to G_\AA \bullet G_\B, \quad 
f(\hs_\CC) = (\hs_\AA, \hs_\B)
\end{equation}
for any $\hs_\CC \in G_\CC$. It follows from \cite[Prop.~14.4.21]{DumF} that $f$ is an isomorphism. This proves the first part of statement (i) of the following lemma. The second part follows directly from \cite[Cor.~14.4.20]{DumF}.
\begin{lem}\label{lem:gal-on-bot}
Let $\CC = \AA\bot \B$ for some modular categories $\AA, \B$, and let  
$$\BF = \BQ(S_\AA) \cap \BQ(S_\B)\,.$$ Then:
\begin{itemize}
\item[(i)] The map $f: G_{\CC} \to G_\AA \bullet G_\B$, $f(\hs_\CC) = (\hs_\AA, \hs_\B)$, defines an isomorphism of groups, and 
\begin{equation}\label{eq:gal_product}
    |G_\CC| = \frac{|G_\AA|\cdot |G_\B|}{[\BF:\BQ]}\,.
\end{equation}
\item[(ii)] For any $\s \in \GQ$, $X \in \irr(\AA)$ and $Y \in \irr(\mathcal{B})$, we have 
$$\hs_{\CC}(X\bot Y) = \hs_{\AA}(X)\bot\hs_{\B}(Y)\,.$$
\item[(iii)] For any $O_\AA \in \Orb(\AA)$ and $O_\B \in \Orb(\B)$, $G_\CC$ acts on $O_\AA \times O_\B$, under the identification of $\irr(\CC) = \irr(\AA) \times \irr(\B)$, and the number of $G_\CC$-orbits in $O_{\AA} \times O_{\B}$ is bounded by $[\BF:\BQ]$. In particular, the numbers of Galois orbits of these categories satisfy $$|\Orb(\AA)|\cdot |\Orb(\B)| \le |\Orb(\CC)| \le |\Orb(\AA)|\cdot|\Orb(\B)|\cdot [\BF:\BQ] \,.$$
\end{itemize}
\end{lem}
\begin{proof} 
The equality \eqref{eq:gal_product} follows directly from \cite[Cor.~14.4.20]{DumF} and the definition of $G_\CC$.

The action of $G_\CC$ on $\irr(\CC)$ is equivalent to the action of $G_\AA \bullet G_\B$ on $\irr(\AA) \times \irr(\B)$ by the definition of the Galois group actions. The statement (ii) follows immediately from this observation. 

To prove (iii), consider any $X \in O_A$ and $Y \in O_\B$. By definition,
$$
\stab_{G_{\CC}}(X\boxtimes Y) \subset \stab_{G_{\AA}}(X) \times \stab_{G_{\B}}(Y)\,.
$$
Therefore, by Burnside's lemma (see, for example, \cite[Ex.~18.3.8]{DumF}) and \eqref{eq:gal_product}, we have 
$$
\begin{aligned}
1 &\le \text{number of } G_{\CC}\text{-orbits in } O_\AA\times O_\B\\
& = \frac{1}{|G_{\CC}|} 
\sum_{(X, Y) \in O_{\AA}\times O_{\B}} |\stab_{G_{\CC}}(X\boxtimes Y)|\\
&\le \frac{1}{|G_{\CC}|} 
\sum_{(X, Y) \in O_{\AA}\times O_{\B}} |\stab_{G_{\AA}}(X)|\cdot |\stab_{G_{\B}}(Y)|\\
&= \frac{|G_{\AA}| \cdot |G_{\B}|}{|G_{\CC}|} 
= [\BF:\BQ]\,.
\end{aligned}
$$
Now, we can establish the last inequalities by summing over all $O_{\AA}\times O_{\B} \in \Orb(\AA)\times \Orb(\B)$.
\end{proof}

\section{Unique factorization of transitive modular categories}\label{sec:sec3}
In this section, we introduce the definition of transitive modular categories. These modular categories have spectacular properties which provide the foundations for the classification. We prove in Theorem~\ref{thm:prime-decomp} that every fusion subcategory of a transitive modular category is a transitive modular subcategory and the prime factorization of a transitive modular category is unique up to permutation 
of prime factors.

\begin{defn}
A modular category $\CC$ is said to be \emph{transitive} if $G_\CC$ (or $\GQ$) acts transitively on $\irr(\CC)$, i.e., $|\Orb(\CC)| = 1$. 
\end{defn}

Recall that a transitive subgroup $G$ of the symmetric group $\mathfrak{S}_n$ is called \emph{regular} if the $G$-action on $\{1, \dots, n\}$ is fixed-point free (cf.~\cite{Wiebook}).

\begin{prop}\label{p:sizeofGalois}
If $\CC$ is a transitive modular category, then $G_\CC$ is regular and $|G_\CC|=|\irr(\CC)|$.
\end{prop}
\begin{proof}
Since $\CC$ is a transitive modular category, $G_\CC$ is  an  abelian transitive subgroup of $\Sym(\irr(\CC))$. By \cite[Prop.\ 4.4]{Wiebook}, $G_\CC$ is regular. In particular,  $|G_\CC|=|\irr(\CC)|$.
\end{proof}

Since $G_\CC$ is regular, for any $X \in \irr(\CC)$, there is a unique $\hs \in G_\CC$ such that $X =\hs(\1)$. Therefore, we simply identify $G_\CC$ with $\irr(\CC)$ via the identification $\hs \mapsto \hs(\1)$. For convenience, we will use $\1$ and $\hat\id$ interchangeably. In particular, the action of $\hs$ on $\hm$ is equal to the product $\hs \hm$ for any $\hs, \hm \in G_\CC$. 

Thus, for any transitive modular category $\CC$, 
%the S- and T-matrices of $\CC$ 
its modular data can be indexed by $G_\CC$. Moreover, the S-matrix can be expressed in terms of the dimensions of simple objects as in the following lemma.

\begin{lem}\label{lem:1}
Let $\CC$ be a transitive modular category. For 
any 
$\hs, \hm \in \irr(\CC)$,  we have
\begin{equation}\label{eq:new-s}
S_{\hs, \hm} = \hs(d_{\hm})d_{\hs} = \hm(d_{\hs})d_{\hm}\,.
\end{equation}
Consequently,  all the entries of the S-matrix  are totally real algebraic units, and every simple object of $\CC$ is self-dual.
\end{lem}
\begin{proof}
Recall from Section \ref{s:galois} that 
$$
\hm\left(\frac{S_{\hs, \1}}{d_\1}\right) = \frac{S_{\hs, \hm}}{d_{\hm}}\,,
$$
so we have $S_{\hs, \hm} = \hm(d_{\hs})d_{\hm}$. Since $S$ is symmetric, we also have $S_{\hs, \hm} = \hs(d_{\hm})d_{\hs}$. 

According to \cite[Prop.~3.6]{RankFinite}, $d_\hs =d_{\hs(\1)}$ is an algebraic unit for all $\hs \in G_\CC$. Since both $ d_{\hs} $ and $ d_{\hm} $ are totally real (cf.~\cite{ENO}),  $S_{\hs, \hm}$ is a totally real unit. In particular, the matrix $ s = \frac{1}{\sqrt{\dim(\CC)}}S $ is a unitary real symmetric matrix, and so we have $\id=s^2 = C$, where $C_{X, Y} = \delta_{X, Y^{\ast}}$ is the charge conjugation matrix (cf.~\cite{BakalovKirillov, ENO}). Therefore, every simple object of $\CC$ is self-dual.
\end{proof}

\begin{cor}\label{cor:FP}
Let $\CC$ be a transitive modular category. Then there exists a unique element $\hs_0 \in G_{\CC}$ such that $\hs_0(d_{\hm}) = \FPdim(\hm)$ for all $\hm \in G_{\CC}$, and 
$$
\hs_0(\dim(\CC)) = \FPdim(\CC)\,.
$$
\end{cor}
\begin{proof}
Since the Frobenius-Perron dimension defines a character of the fusion ring $K_0(\CC)$ and 
the simple objects are in one-to-one correspondence to the characters of the fusion ring (see Section \ref{s:galois}), there exists a unique simple object $\hs_0 \in G_{\CC}$ such that $\chi_{\hs_{0}}(\hm) = \FPdim(\hm)$ for all $\hm\in G_{\CC}$. Therefore, by Lemma \ref{lem:1}, we have
$$
\FPdim(\hm) = \chi_{\hs_{0}}(\hm) 
= 
\frac{S_{\hm,\hs_{0}}}{d_{\hs_{0}}}
=
\hs_{0}(d_{\hm})
\,.
$$
The second assertion follows directly from the first statement and the definitions of $\dim(\CC)$ and $\FPdim(\CC)$ .
\end{proof}

Now, we can prove the first major observation on transitive modular categories.

\begin{thm}\label{thm:1}
Let $\CC$ be a transitive modular category. Then:
\begin{enumerate}
\item[{\rm (i)}] 
For any $\hs, \hm \in \irr(\CC)$, if  $\hs \neq \hm$, then $d_{\hs}^2 \neq d_{\hm}^2$. In particular, if $\hm \neq \hat\id$, then 
$$
d_{\hm}^2 \neq 1 \quad \text{and}\quad \hm(\dim(\CC)) \neq \dim(\CC).
$$
\item[\rm (ii)] 
If $X$ is an invertible object in $\CC$, then $X \cong \1$.  In particular, $\CC_\pt \simeq \Vec$ as fusion categories.
\item[\rm (iii)] 
$\BQ(S)=\BQ(\dim(\CC))= \BQ(d_X \mid X \in \irr(\CC))$.
\end{enumerate}
\end{thm}
\begin{proof}
Suppose there exist $\hs\neq\hm \in \irr(\CC)$ such that $d_{\hs}^2 = d_{\hm}^2$. Then  $d_{\hs} = \e d_{\hm}$ for some $\e\in\{\pm 1\}$. By Lemma \ref{lem:1}, for any $\hld \in \irr(\CC)$, we have 
$$
S_{\hs, \hld} = \hld(d_{\hs})d_{\hld} = \hld(\e d_{\hm})d_{\hld} = \e\hld(d_{\hm})d_{\hld} = \e S_{\hm, \hld}\,.
$$
Consequently, the rows $S_{\hs, *}$ and $S_{\hm, *}$ of $S$ are linearly dependent, which contradicts the invertibility of the S-matrix. This proves the first assertion of statement (i).

Note that $d_{\hat\id} = d_{\1} = 1$. Therefore, for any $\hm \neq \hat{\id}$, we have $d_{\hm}^2 \neq 1$. In particular, up to isomorphism, there is no other invertible object in $\CC$ than $\1$, which implies statement (ii). Moreover, by \eqref{eq:sgn-1}, we find
\begin{equation}\label{eq:1}
\hm\left(\frac{d_{\1}^2}{\dim(\CC)}\right) 
=
\hm\left(\frac{1}{\dim(\CC)}\right) 
= 
\frac{d_{\hm}^2}{\dim(\CC)}\,.
\end{equation}
Hence, 
$$
\frac{\dim(\CC)}{\hm(\dim(\CC))} = d_{\hm}^2 \neq 1\,,
$$
and we have completed the proof of statement (i). 

Since $\BQ(\dim(\CC))$ is a subfield of $\BQ(S)$, it is abelian and hence Galois over $\BQ$. By (i), there is no nontrivial element of $G_{\CC}$ fixing $\dim(\CC)$. By the fundamental theorem of Galois theory, $\BQ(\dim(\CC)) = \BQ(S)$. By the definition of $\BQ(S)$, we always have the inclusions
$$
\BQ(\dim(\CC))  \subseteq \BQ(d_\hm \mid \hm \in G_\CC) \subseteq \BQ(S)\,.
$$
The equality $\BQ(\dim(\CC)) = \BQ(S)$ implies $\BQ(S)=  \BQ(d_\hm \mid \hm \in G_\CC)$.
\end{proof}

\begin{cor}\label{cor:unique-sph}
If $\CC$ is a transitive modular category, then the underlying braided fusion category has a unique  pivotal structure up to isomorphism.
\end{cor}
\begin{proof}
By \cite[Lem. 2.4]{RankFinite}, there is a bijective correspondence between $\irr(\CC_{\pt})$ and isomorphism classes of pivotal structures of the underlying fusion category of $\CC$. By Theorem \ref{thm:1} (ii),  $\irr(\CC_{\pt})$ is trivial since $\CC$ is transitive. Therefore, the underlying pivotal structure of the modular category $\CC$ is the only one up to isomorphism.
\end{proof}

Recall that a fusion category $\CC$ is called \emph{weakly integral} if $\FPdim(\CC) \in \BZ$, and is said to be \emph{trivial} if it is tensor equivalent to $\Vec$. 

\begin{cor}\label{cor:int-cat}
If $\CC$ is a transitive modular category and $\DD \subset \CC$ a nontrivial fusion subcategory, then $\dim(\DD) \not\in \BZ$. In particular, $\CC$ does not contain any nontrivial weakly integral fusion subcategories.
\end{cor}
\begin{proof}
Suppose $\DD$ is a fusion subcategory of $\CC$ such that $\dim(\DD) \in \BZ$.  Let $\hs_0\in G_{\CC}$ be the canonical element realizing the Frobenius-Perron dimension in Corollary \ref{cor:FP}. Then $\hs_0(\dim(\DD)) = \FPdim(\DD)$ and hence  $\FPdim(\DD) \in \BZ$. In other words, $\DD$ is weakly integral. By \cite[Prop.\ 8.27]{ENO}, for any $\hm \in \irr(\DD)$, $\hs_0(d_\hm^2) = \FPdim(\hm)^2 \in \BZ$. Therefore, $d_{\hm}^2 \in \BZ$. By Lemma \ref{lem:1}, $d_{\hm}$ is a real algebraic unit for any $\hm \in G_{\CC}$, and so $d_\hm^2 = 1$. However, by Theorem 3.5, this means $\hm =\hat\id$ and hence $\irr(\DD) = \{\1\}$. This proves the first statement of the corollary.

Note that every weakly integral fusion category $\B$ satisfies $\FPdim(\B)=\dim(\B) \in \BZ$ (cf.~\cite{ENO}). Therefore, if $\B$ is a weakly integral fusion subcategory of $\CC$, then  $\B$ must be trivial by the preceding assertion. 
\end{proof}

\begin{rmk}
As a consequence of Corollary \ref{cor:int-cat}, if $\CC$ is a transitive modular category satisfying $\dim(\CC) \in \BZ$, then $\CC$ is trivial.
\end{rmk}

In the following, we study fusion subcategories and Deligne products of transitive modular categories.

\begin{cor}\label{cor:fus-subcat}
Every fusion subcategory of a transitive modular category $\CC$ is a modular subcategory of $\CC$. 
\end{cor}
\begin{proof}
Let $\DD$ be a fusion subcategory of $\CC$. Then $\DD$ is premodular with the braiding and the spherical pivotal structure inherited from $\CC$. Now consider the M\"uger center $\DD' = C_{\DD}(\DD)$ of $\DD$. It is a symmetric fusion subcategory of $\DD$ and hence of $\CC$. Then, by \cite{Del1,Del2}, there exists a finite group $H$ such that $\DD'$ is tensor equivalent to $\Rep(H)$. In particular, $\dim(\DD')=|H|$ is an integer. By Corollary \ref{cor:int-cat}, $\DD'$ is equivalent  $\Vec$ as a fusion category. Therefore, $\DD$ is modular.
\end{proof}

\begin{example}\label{ex:Fib-2}
The quantum group modular category $\CC=\CC(\sl_2,3)^{(0)}$ is a Fibonacci modular category with two isomorphism classes of simple objects
$\1$ and $\tau$ such that $\tau \otimes \tau = \1 \oplus \tau$ (cf.~\cite{SRW}).  The S-matrix of $\CC$ is given by
$$
S=\begin{pmatrix}
1 & d_\tau\\
d_{\tau} & -1
\end{pmatrix}$$ 
where $d_\tau = \frac{1+\sqrt{5}}{2}$. Therefore, $\BQ(S)=\BQ(\sqrt{5})$ and $G_\CC \cong \BZ_2$ with the generator $\hs: \sqrt{5} \mapsto -\sqrt{5}$. Therefore, $\CC$ is transitive.  
\end{example}

Recall that a modular category $\CC$ is \emph{prime} if every modular subcategory 
of $\CC$  
is equivalent to $\CC$ or $\Vec$. By \cite[Thm.~4.5]{Muger-Structure}, every modular category admits a \emph{prime factorization}, i.e., it is equivalent to a finite Deligne product of prime modular categories.

\begin{thm}\label{thm:prime-decomp}
Let $\CC$ be a transitive modular category. Then:
\begin{itemize}
\item[\rm (i)] Every fusion subcategory of $\CC$ is a transitive modular subcategory, and 
\item[\rm (ii)] the prime factorization of $\CC$ is unique up to permutation of factors.
\end{itemize}
\end{thm}
\begin{proof}
Let $\AA$ be an arbitrary fusion subcategory of $\CC$. It follows from Corollary \ref{cor:fus-subcat} that $\AA$ is a modular subcategory of $\CC$. By the double centralizer theorem \cite[Thm.~4.2]{Muger-Structure}, $\B := C_{\CC}(\AA)$, the M\"uger centralizer of $\AA$ in $\CC$, is modular. Moreover, there is an equivalence of modular categories
$$\CC\simeq\AA \bot \B\,.$$ 

As noted in Section \ref{s:galois}, for any $X \boxtimes Y \in \irr(\CC) = \irr(\AA\boxtimes \B)$ and any $\s \in \GQ$, we have $\hs_{\CC}(X\boxtimes Y) = \hs_{\AA}(X) \boxtimes \hs_{\B}(Y)$. Since $\CC$ is transitive, for any $X \in \irr(\AA)$, there exists $\s \in \GQ$ such that 
$$
\hs_\AA(\1_{\AA}) \bot \hs_\B(\1_{\B}) = \hs_{\CC}(\1_{\AA}\bot \1_{\B})  = X \bot \1_{\B}\,.
$$ 
Therefore, $G_\AA$ acts transitively on $\irr(\AA)$.  This completes the proof of (i).

It follows from \cite[Thm.~4.5]{Muger-Structure} that $\CC$ admits  a prime factorization. By Corollary \ref{cor:int-cat}, $\CC_{\pt} \simeq \Vec{}$. Therefore, by \cite[Prop.~2.2]{DMNO}, the prime factorization of $\CC$ is unique up to permutation of factors.
\end{proof}

\begin{prop}\label{prop:trans-prod}
If $\AA, \B$ are transitive modular categories and $\CC = \AA\bot\B$, then \begin{equation}\label{eq:orb-AB}
|\Orb(\CC)| 
= 
\frac{|G_{\AA}|\cdot |G_{\B}|}{|G_{\CC}|}
=
[\BQ(\dim(\AA)) \cap \BQ(\dim(\B)): \BQ]\,.
\end{equation}
In particular, if $\AA, \B$ are nontrivial modular categories and they are Galois conjugate to each other, then $\AA \bot \B$ is not transitive.
\end{prop}
\begin{proof}
Since $\AA, \B$ are transitive, as discussed at the beginning of this section, the action of $G_{\AA}$ (resp.~$G_{\B}$) on $\irr(\AA)=G_\AA$ (resp.~$G_{\B} = \irr(\B)$) is just the left multiplication. By Lemma~\ref{lem:gal-on-bot}, $G_\AA\bullet G_\B \cong G_{\CC}$ is a subgroup of $G_{\AA} \times G_{\B}$, and the action of $G_{\CC}$ on $\irr(\CC) = \irr(\AA) \times \irr(\B) = G_{\AA} \times G_{\B}$ is equivalent to the left multiplication by $G_\AA\bullet G_\B$. Therefore, the orbits of this $G_\AA\bullet G_\B$-action are the cosets of $G_\AA\bullet G_\B$ in $G_{\AA}\times G_{\B}$, which implies the first equality in \eqref{eq:orb-AB}. The second equality in \eqref{eq:orb-AB} is a direct application of Lemma~\ref{lem:gal-on-bot}(i) and Theorem~\ref{thm:1}(iii).

Suppose $\AA$ and $\B$ are nontrivial Galois conjugate modular categories. Then 
 $\BQ(\dim(\AA))=\BQ(\dim(\B))$  and  $\BQ(\dim(\AA))$ is a proper extension of $\BQ$ (cf.~Theorem \ref{thm:1}). Therefore, we have 
 $$
 |\Orb(\AA \bot \B)| = [\BQ(\dim(\AA)) \cap \BQ(\dim(\B)): \BQ] = [\BQ(\dim(\AA)): \BQ] > 1,
$$  
which means $\AA \bot \B$ is not transitive. This completes the proof of the last assertion.
\end{proof}

The following corollary provides a necessary and sufficient condition for 
the
transitivity of a Deligne product. 

\begin{cor}\label{cor:coprime}
Let $\CC$, $\DD$ be modular categories. Then $\CC\bot\DD$ is transitive if and only if the following two conditions hold: both $\CC$, $\DD$ are transitive and $$\BQ(\dim(\CC)) \cap \BQ(\dim(\DD)) = \BQ\,.$$
\end{cor}
\begin{proof}
If $\CC\bot\DD$ is transitive, then, by Theorem \ref{thm:prime-decomp}, $\CC$ and $\DD$ are also transitive. By Proposition~\ref{prop:trans-prod}, $1 = |\Orb(\CC\bot\DD)| = [\BQ(\dim(\CC)) \cap \BQ(\dim(\DD)):\BQ]$, and  so $\BQ(\dim(\CC)) \cap \BQ(\dim(\DD)) = \BQ$. 

Conversely, assume $\CC$ and $\DD$ are transitive modular categories and $\BQ(\dim(\CC)) \cap \BQ(\dim(\DD)) = \BQ$, then $[\BQ(\dim(\CC)) \cap \BQ(\dim(\DD)): \BQ] = 1$. It follows from Proposition~\ref{prop:trans-prod} that $\CC\bot\DD$ is transitive.
\end{proof}

\section{Primality of transitive quantum group modular categories}\label{sec:example}

A quantum group modular category $\CC(\mathfrak{g}, k)$ can be constructed from a simple Lie algebra $\mathfrak{g}$ and a positive integer $k$, which is called the \emph{level}. This modular category is a semisimplification of the tilting module category of the quantum group $U_q(\mathfrak{g})$ specialized at a root of unity $q$ determined by $k$ and $\mathfrak{g}$. The readers are referred to \cite{BakalovKirillov, Row05} and the references therein for details.

In this paper, we focus on the cases when $\mathfrak{g} = \sl_2$. Let $k$ a positive integer and $q=\exp\left(\frac{\pi i }{k+2}\right)$. For any $r \in \BQ$, we define
$$
q^r := \exp\left(\frac{\pi i r }{k+2}\right)\,.
$$
The quantum integer $[n]_\zeta$ for any root of unity $\zeta \ne \pm 1$ is defined as
$$
[n]_\zeta := \frac{\zeta^n - \zeta^{-n}}{\zeta-\zeta^{-1}}\,.
$$
The isomorphism classes of simple objects of $\CC(\sl_2, k)$ are indexed by the integers $a \in [0,k]$. The modular data $(S, T)$ of the modular category $\CC(\sl_2, k)$ is given by (cf. \cite{BakalovKirillov}, see also \cite{RT} with a different convention)
\begin{equation}\label{eq:su2-st}
S_{a, b} = [(a+1)(b+1)]_q\,,
\quad
T_{a, b} = \delta_{a,b}\, q^{a(a+2)/2}\,,
\quad
0 \le a, b \le k\,.
\end{equation}
One can replace $q$ by any Galois conjugate $q' = q^l$ for some $l$ relatively prime to $2(k+2)$ to get another modular category $\CC(\sl_2, k, q^l)$. The simple objects of this modular category are also indexed by the integers in $[0,k]$ and its modular data is also given by \eqref{eq:su2-st} with $q$ replaced by $q^l$. 

For the discussions of the remainder of this paper, we will simply write $\AA_{k,l}$ for the modular category $\CC(\sl_2, k, q^l)$ where $\gcd(l, 2(k+2))=1$. Let $V_a$ denote the isomorphism class of the simple objects of $\AA_{k,l}$ indexed by the integer $a \in [0,k]$. Then $V_0$ is the isomorphism class of the tensor unit $\1$. The fusion rules of $\AA_{k,l}$ are the same for any possible integer $l$, and they are given by (cf.~\cite{BakalovKirillov})
\begin{equation}\label{eq:su2}
N_{a,b}^c = 
\begin{cases}
1, & \text{ if } |a-b| \le c \le \min(a+b, 2k-a-b) \\
    & \qquad\qquad\qquad\qquad \qquad \text{ and } c \equiv a+b\pmod{2}\,;\\
0, & \text{ otherwise.}
\end{cases}
\end{equation}

One can observe directly from the fusion rules that $\AA_{k,l}$ is $\BZ_2$-graded, where the homogeneous component $A_{k,l}^{(j)}$, $j\in \{0,1\}$, is the $\BC$-linear subcategory (additively) generated by the simple objects $V_a$ satisfying $a \equiv j \pmod{2}$ for any integer $a \in [0,k]$. Moreover, the adjoint fusion subcategory of $\AA_{k,l}$ is $\AA_{k,l}^{(0)}$, which is a modular subcategory of $\AA_{k,l}$ if and only if $k$ is odd (cf.~\cite{Bruguieres, TurWen}), and 
\begin{equation}\label{eq:irr-B}
\irr(\AA^{(0)}_{k,l}) = \left\{V_{2j}\mid 0 \le j \le  \frac{k-1}{2}\right\}\,.
\end{equation}
In particular, when $k=1$, $\BB_{1,l}$ is tensor equivalent to $\Vec{}$, and when $k = 3$, $\BB_{3,l}$ is a Fibonacci modular category (see Example~\ref{ex:Fib-2}). 

For any fusion category $\CC$, we say that a simple object $X \in\irr(\CC)$ \emph{tensor generates} $\CC$ if every simple object of $\CC$ is isomorphic to a  summand of a tensor power of $X$. 
The following observation could be known to experts but  we include it here for completeness.

\begin{lem}\label{lem:gen}
For any positive odd integer $k$ and $l \in \left(\Zn{2(k+2)}\right)^\times$, every nontrivial simple object of $\BB_{k,l}$ tensor generates $\BB_{k,l}$. In particular, $\BB_{k,l}$ is a prime modular category.
\end{lem}
\begin{proof}
Since $\BB_{1,l}$ is trivial, the statements are true for $k = 1$. We assume $k \ge 3$. By the fusion rules \eqref{eq:su2}, when $k=3$, $V_2 \o V_2 = V_0 \oplus V_2$; when $k > 3$ is odd, for any $1 \le j \le \frac{k-3}{2}$, we have
$$V_{2j} \otimes V_2 = V_{2j-2} \oplus V_{2j} \oplus V_{2j+2}\,.$$ 
Therefore, $V_2$ tensor generates $\BB_{k,l}$. Moreover, for any $1 \le j \le \frac{k-1}{2}$, we have $2 \le \min(4j, 2k-4j)$, so $N_{2j, 2j}^{2} = 1$, which means $V_2$ is a direct summand of $V_{2j} \o V_{2j}$. Therefore, $V_{2j}$ tensor generates $\BB_{k,l}$.
\end{proof}

For any odd integer $k$ and  $l \in \left(\Zn{2(k+2)}\right)^\times$, the  modular data $(S^{(0)}, T^{(0)})$ of $\AA_{k,l}^{(0)}$ is indexed by $j=0,\dots, \frac{k-1}{2}$, and is given by
\begin{equation}\label{eq:so3-st}
S^{(0)}_{j, m} = [(2j+1)(2m+1)]_{q^l}, 
\quad 
T^{(0)}_{j, m} = \delta_{m,j} q^{2lj(j+1)}, 
\quad 
0 \le j, m \le \frac{k-1}{2}\,.
\end{equation}
It is well-known that the first central charge of $\AA_{k,1}$ is given by (cf.~\cite{BakalovKirillov})
$$
\xi_1(\AA_{k,1}) = \exp\left(\frac{3k\pi i}{4(k+2)}\right)\,.
$$
By definition (see \eqref{eq:anomaly}), the first anomaly of $\AA_{k,1}$ is   
$$
\alpha_1(\AA_{k,1}) = \xi_1(\AA_{k,1})^2 = \exp\left(\frac{3k\pi i}{2(k+2)}\right)\,.
$$
By the fusion rules, $\irr((\AA_{k,1})_{\pt})= \{V_0, V_k\}$, and $(\AA_{k,1})_{\pt}$ is a modular subcategory of $\AA_{k,1}$. By \cite[Cor.~3.27]{DrGNO},  $C_{\AA_{k,1}}((\AA_{k,1})_{\pt}) = \BB_{k,1}$. Therefore,  
$$
\AA_{k,1} \simeq \BB_{k,1} \boxtimes (\AA_{k,1})_{\pt}
$$
as modular categories by the double centralizer theorem \cite[Thm.~4.2]{Muger-Structure}. Consequently, by \cite[Lemma 3.12]{NSW}, $\alpha_1(\AA_{k,1}) = \alpha_1(\BB_{k,1})\cdot\alpha_1((\AA_{k,1})_{\pt})$.
Following \eqref{eq:su2-st}, we have
$$
\alpha_1((\AA_{k,1})_{\pt}) = \frac{1+i^k}{1-i^k} = i^k\,.
$$
Therefore,
\begin{equation}\label{eq:alpha-Bk1}
\alpha_1(\BB_{k,1}) 
= 
\frac{\alpha_1(\AA_{k,1})}{\alpha_1((\AA_{k,1})_{\pt})}
=
\exp\left(\frac{(1-k)k\pi i}{2(k+2)}\right)\,.
\end{equation}

In the literature, $\BB_{k,1}$ is often referred to as the quantum group modular category  ``$\SO(3)$ at level $k$'' or ``PSU(2) at level $k$''. The ribbon categories with these fusion rules for odd $k$ are completely classified in \cite[Cor.~8.2.7]{FK}, with a slightly different parametrization. 

\begin{lem}\label{lem:FK-so3}
For any positive odd integer $k$, the modular categories 
$$
\BB_{k,l}, \quad l \in \left(\frac{\BZ}{{2(k+2)\BZ}}\right)^\times
$$ 
form a complete list of inequivalent ribbon categories with the fusion rules of $\SO(3)$ at level $k$. If $k+2=p>3$ is a prime, each of these modular categories is equivalent to a Galois conjugate of $\BB_{p-2,1} = \CC(\sl_2, p-2)^{(0)}$.
\end{lem}
\begin{proof}
The first part follows directly from \cite[Cor.~8.2.7]{FK}. If $k+2=p>3$ is a prime, there are exactly $|(\BZ/2p\BZ)^\times|=p-1$ equivalence classes of  ribbon categories with the fusion rules of $\SO(3)$ at level $k$. Note that all Galois conjugates of $\BB_{k,1}$ have the same fusion rules. Hence, they are equivalent to the modular categories in the list. According to \eqref{eq:alpha-Bk1}, $\alpha_1(\BB_{p-2,1}) = \exp\left(\frac{(3-p)(p-2)\pi i}{2p}\right) \in \BQ_p$, which is a root of unity of order $p$ or $2p$ for $p>3$. By definition, the first anomalies of the Galois conjugates of $\BB_{k,1}$ are the Galois conjugates of $\alpha_1(\BB_{k,1})$. Therefore, there are at least $\varphi(p)  = p-1$ equivalence classes among the Galois conjugates of $\BB_{k,1}$, and we are done by the first assertion.
\end{proof}

Now, we can show a family of these quantum group modular categories are prime and transitive in the following proposition. 
\begin{prop} \label{p:B-trans-prime}
Let $p$ be any odd prime and $l \in \UZ{2p}$. Then the modular category $\AA^{(0)}_{p-2, l}$ is prime and transitive.
\end{prop}
\begin{proof}
By Lemma~\ref{lem:gen}, $\BB_{p-2,l}$ is prime. Therefore, it suffices to show that $\BB_{p-2,l}$ is transitive. 

The underlying root of unity  $q^l$ is the primitive $2p$-th root of unity. Since $p$ is odd, $\BQ_p = \BQ(q^l)= \BQ(q^{2l})$, it suffices to show that $\gal(\BQ_p/\BQ)$ acts transitively on $\irr(\AA^{(0)}_{p-2,l})$. 

For any nonnegative integer  $m \le \frac{p-3}{2}$, $\gcd(2m+1, 2p) = 1$. So there exists $\s \in  \gal(\BQ_p/\BQ)$ such that $\s(q) = q^{2m+1}$. Thus, we have
$$
\s\left(\frac{S^{(0)}_{j,0}}{S^{(0)}_{0,0}}\right) 
= 
\s([(2j+1)]_{q^l}) 
= 
\frac{[(2j+1)(2m+1)]_{q^l}}{ [2m+1]_{q^l}} 
= 
\frac{S^{(0)}_{j, m}}{S^{(0)}_{0,m}}
$$
for any nonnegative integer $j \le \frac{p-3}{2}$. Therefore, $\hs(V_0)=V_{2m}$ and hence $\gal(\BQ_p/\BQ)$ acts transitively on $\irr(\AA^{(0)}_{p-2,l})$.
\end{proof}

\begin{prop}\label{p:prod_A}
Let $p_1, \dots, p_\ell > 3$ be distinct primes.  For any $(l_1, \dots, l_\ell) \in \UZ{2p_1}\times \cdots \times \UZ{2p_\ell}$,
the Deligne product 
$$
\CC = \BB_{p_1 -2, l_1} \bot \cdots \bot \BB_{p_\ell -2, l_\ell}
$$
is a transitive modular category.
\end{prop}
\begin{proof}
We proceed to prove the statement by induction on $\ell$.  The statement obviously holds for $\ell=1$ by Proposition \ref{p:B-trans-prime}. Now we assume  $p_1, \dots, p_{\ell} > 3$ are distinct primes and $(l_1, \dots, l_{\ell}) \in \UZ{2p_1}\times \cdots \times \UZ{2p_{\ell}}$ for some integer $\ell >1$.   By the induction assumption,   $\CC=\boxtimes_{a = 1}^{\ell-1}\BB_{p_{a}-2, l_{a}}$ is a transitive modular category. Note that
$$
\dim(\CC) = \prod_{a = 1}^{\ell-1}\dim(\BB_{p_{a}-2, l_{a}}) \in \BQ_{p'_\ell} \quad\text{and}\quad
\dim(\BB_{p_{\ell}-2, l_{\ell}}) \in \BQ_{p_{\ell}},
$$
where $p'_\ell = p_1 \cdots p_{\ell-1}$. Since $\BQ_{p'_\ell} \cap \BQ_{p_{\ell}} = \BQ$, it follows from Corollary \ref{cor:coprime} and Proposition \ref{p:B-trans-prime} that $\CC \bot \BB_{p_{\ell}-2, l_{\ell}}$ is transitive. 
\end{proof}

\section{Representations of \texorpdfstring{$\SLZ$}{} associated with modular categories} \label{subsec:mod-group-rep}
In this section, we show that the representations of $\SLZ$ associated with transitive modular categories are irreducible and minimal. As a consequence, the order of the T-matrix of any nontrivial transitive modular category is square-free and its prime factors are greater than 3.

Let $\CC$ be a modular category with modular data $(S,T)$. We denote by $\GL_\CC(\BC)$ the group of all invertible matrices over $\BC$ indexed by $\irr(\CC)$, and $V_\CC=K_0(\CC)\o_\BZ \BC$ with the  standard basis $E_\CC=\{e_X\mid X \in \irr(\CC)\}$. Note that $S,T \in \GL_\CC(\BC)$, and the group $\GL_\CC(\BC)$ acts on $V_\CC$ via the standard basis $E_\CC$, namely $A(e_Y) = \sum_{X \in \irr(\CC)} A_{XY} e_X$ for any $A \in \GL_\CC(\BC)$ and $Y \in \irr(\CC)$.
We often identify $\GL(V_\CC)$ with $\GL_\CC(\BC)$ in this manner.

Recall that 
$\fs := 
\begin{pmatrix}
0 & -1\\ 1 & 0
\end{pmatrix}$ 
and 
$\ft := 
\begin{pmatrix}
1 & 1 \\ 0 & 1
\end{pmatrix}$ 
are generators of the group $\SL_2(\BZ)$, subjected to the relations $\fs^4 = \id$ and $(\fs\ft)^3 = \fs^2$, and the assignment
\begin{equation}
    \bar{\rho}_{\CC}: \SL_2(\BZ) \to \PGL(V_\CC),
    \quad
    \fs \mapsto S,
    \quad
    \ft \mapsto T
\end{equation}
defines a group homomorphism (cf.~\cite{BakalovKirillov, TuBook}). This projective representation $\ol\rho_\CC$ can  be lifted to an ordinary representation  $(\rho, V_\CC)$ such that the diagram 
$$
\xymatrix{\SL_2(\BZ)\ar[rd]_-{\ol\rho_\CC} \ar[r]^-{\rho} & \GL(V_\CC) \ar[d] \\
& \PGL(V_\CC) 
}
$$
commutes, where the vertical map $\GL(V_\CC) \to \PGL(V_\CC)$ is the natural surjection. 
Any lifting $(\rho, V_\CC)$ of $\ol\rho_\CC$, called a \emph{representation of $\SLZ$ associated with $\CC$},  yields an action of $\SLZ$ on $V_\CC$ given by 
$$
\fa \cdot e_Y :=  \rho(\fa) (e_{Y}) =\sum_{X \in \irr(\CC)} \rho(\fa)_{X,Y} (e_X) 
$$
for any $\fa \in \SLZ$. We call $V_\CC$ an \emph{$\SLZ$-module of $\CC$} throughout this paper. If $(\rho, V_\CC)$ is a representation of $\SLZ$ associated with $\CC$, then the pair $(s,t):=(\rho(\fs), \rho(\ft))$, called the \emph{normalized modular data}, uniquely  determines $\rho$, and the matrices $s, t$ are unitary and symmetric (cf.~\cite{ENO}). Moreover, the group of 1-dimensional representations of $\SLZ$ acts transitively on representations of $\SLZ$ associated with $\CC$ by tensor product (cf.~\cite{DongLinNg}).

For any positive integer $m$, we denote by $\pi_m: \SLZ \to \SL_2(\Zn{m})$ the natural surjection. We say that a representation $\phi: \SLZ\to \GL_r(\BC)$ is of level $m$ if $\phi = \tilde\phi\circ \pi_m$ for some representation $\tilde\phi: \SL_2(\Zn{m})\to \GL_r(\BC)$ and $m=\ord(\phi(\ft))$.  By \cite[Thm.\ II]{DongLinNg}, if  $\rho$ is a representation of  $\SLZ$ associated with $\CC$, then $\rho$ is of level $n= \ord(\rho(\ft))$ and $\rho(\fa)_{X,Y} \in \BQ_n$ for any $\fa \in \SLZ$ and $X,Y \in \irr(\CC)$. In particular, $s, t$ are matrices defined over $\BQ_n$. Thus, for  $\s \in \gal(\BQ_n/\BQ)$, $(^\s\!\rho, V_\CC)$ is also a representation of $\SLZ$ where $^\s\!\rho(\fa) = \s(\rho(\fa))$ for any $\fa \in \SLZ$, and the corresponding $\s$-twisted $\SLZ$-action on $V_\CC$ is denoted by
\begin{equation}\label{eq:twisted-action}
    ^\s\!\fa \cdot v = {^\s\!\rho}(\fa)(v)  
\end{equation}
for any $v \in V_\CC$.

Let $(\rho, V_{\CC})$ be a level $n$ representation of $\SLZ$ associated with a modular category $\CC$. The action of the Galois group $\gal(\BQ_n/\BQ)$ on the normalized modular data $(s,t)$ satisfies some interesting conditions as follows: for  $\s \in \gal(\BQ_n/\BQ)$, 
there exists a sign function $\e_\s: \irr(\CC) \to \{\pm 1\}$ such that 
\begin{equation}\label{eq:the-sgn}
\s(s_{X,Y}) = \e_\s(X) s_{\hs(X), Y} = \e_\s(Y) s_{X, \hs(Y)}
\end{equation}
for any $X,Y \in \irr(\CC)$(cf.~\cite{CosteGannon, dBG}), and 
\begin{equation}\label{eq:gal_t}
\s^2(t_{X,X}) = t_{\hs(X), \hs(X)}
\end{equation}
for any $X \in \irr(\CC)$ (cf.~\cite[Thm.\ II (iii)]{DongLinNg}). Moreover, the absolute Galois group $\GQ$  acts on the normalized modular data via the restriction 
$$\res^{\ol\BQ }_{\BQ_n}:\GQ \to \gal(\BQ_n/\BQ)\,.$$ 

The condition of the action of $\GQ$ on $s$ defines a $\GQ$-action on $V_\CC$.  Let  $g_\s \in \GL(V_\CC)$ be defined by 
\begin{equation*}
(g_{\s})_{X,Y} := \e_\s(X)\, \delta_{\hs(X), Y}\,.
\end{equation*}
Then, \eqref{eq:the-sgn} and \eqref{eq:gal_t} can be rewritten as
\begin{equation}\label{eq:Gs}
\s(s) = g_{\s} s = s  g_\s^{-1}, \quad \s^2(t) = g_\s t g_\s^{-1}\,,
\end{equation}
and the assignment
\begin{equation}\label{eq:phi-rho}
\phi_\rho: \GQ \to \GL(V_\CC), \quad \s \mapsto g_{\s},
\end{equation}
defines a group homomorphism (cf.~\cite{CosteGannon}). Therefore, for any $\s\in \GQ$, we have
\begin{equation}\label{eq:conjuate-rep}
^{\s^2}\!\!\!\rho(\fa) = g_\s \rho(\fa) g_\s^{-1} \quad \text{for all }\fa \in \SLZ\,.
\end{equation}
In particular,  $(\rho, V_\CC) \cong (^{\s^2}\!\!\!\rho, V_\CC)$ as  representations of $\SLZ$. 

Now, $\GQ$ acts on $V_\CC$ via the representation $(\phi_\rho, V_\CC)$ of $\GQ$, namely
\begin{equation} \label{eq:GQ-action}
    \s \cdot e_X := g_\s(e_X) =  \e_\s(X) \ e_{\hs(X)}
\end{equation}
for any $\s \in \GQ$ and $X \in \irr(\CC)$. Thus, in view of \cite[Thm.~II (iii)]{DongLinNg} or \eqref{eq:conjuate-rep}, for any $\fa \in \SLZ$, $v \in V_\CC$ and $\s \in \GQ$, we have
\begin{equation} \label{eq:GQ-SL-relation}
 \s \cdot (\fa \cdot v)  = g_\s \rho(\fa) (v) =  g_\s \rho(\fa) g_\s^{-1} g_\s (v) = {^{\s^2}\!\!\fa} \cdot (\s \cdot v)\,.
\end{equation}
By \cite[Thm.~II (iv)]{DongLinNg}, if $\s(\zeta_n) = \zeta_n^a$ for some integer $a$ coprime to $n$, then 
\begin{equation} \label{eq: g_sigma}
g_\s = \rho(\ft^a \fs \ft^b \fs \ft^a \fs^{-1})\,, 
\end{equation}
where $b$ is an inverse of $a$ modulo $n$. Therefore, the $\GQ$-action on $V_\CC$ is uniquely determined by $\rho$, and every $\SLZ$-submodule of $V_\CC$ also inherits the action of $\GQ$. 

\subsection{Minimal representations of \texorpdfstring{$\SLZ$}{}}
To proceed, we set up the following conventions. We will denote by $\spec{M}$ the set of the eigenvalues of an linear operator $M$ on a finite-dimensional complex vector space. For any finite multiplicative abelian group $A$, 
$$A^2 := \{a^2 \mid a \in A\}$$
is a subgroup of $A$ of order $|A|/|\Omega_2(A)|$, where $\Omega_2(A)$ is the elementary 2-subgroup of $A$. In particular, for any positive integer $m$, $\Omega_2(\gal(\BQ_m/\BQ))$ is simply denoted by $\Omega_2^m$ and we define
$$
\varphi_2(m) := \left| ((\Zn{m})^\times)^2\right| = \left| \gal(\BQ_m/\BQ)^2\right|\,.
$$
It is immediately seen that $\varphi_2$ is a multiplicative function. Moreover, for any prime $p$, we have
\begin{equation}\label{eq:phi2}
\varphi_2(p^m) = \left\{\begin{array}{ll}
\frac{1}{2} (p-1)p^{m-1} & \text{ if  $p$ is odd}; \\
2^{m-3} & \text{ if  $p=2$ and $m \ge 3$};\\
1 & \text{ if  $p=2$ and $m =1,2$}.\\
\end{array}\right.
\end{equation}

Suppose $(s,t)$ is a normalized modular data of a modular category $\CC$.  By \eqref{eq:gal_t}, the assignment $(\s, \zeta)\mapsto \s^2(\zeta)$ defines a $\GQ$-action on $\spec{t}$, and the $\GQ$-orbit of $t_{X,X}$, for any $X \in \irr(\CC)$, is then given by
$$
\GQ \cdot t_{X,X} =\{\s^2(t_{X,X})\mid \s \in  \gal(\BQ_m/\BQ)\}
$$
where $m = \ord(t_{X,X})$. In particular, $|\GQ \cdot t_{X,X}|=\varphi_2(m)$. We denote by $\spec{t}/\GQ$ the set of $\GQ$-orbits of $\spec{t}$.

\begin{lem}
Let $(\rho, V_\CC)$ be a representation of $\SLZ$ associated with a modular category $\CC$. If $(\rho|_W, W)$ is a subrepresentation of $(\rho, V_\CC)$, then $\spec{\rho(\ft)|_W}$ is closed under the action of $\GQ$ on $\spec{\rho(\ft)}$. In particular, 
$$
\spec{\rho(\ft)|_W}/\GQ \subseteq \spec{\rho(\ft)}/\GQ,
$$
and every direct sum decomposition of $(\rho, V_\CC)$ as representations of $\SLZ$ determines   a partition of $\spec{\rho(\ft)}/\GQ$.
\end{lem}
\begin{proof}
 For any $\zeta \in \spec{\rho(\ft)}$,  $B_\zeta = \{e_X \mid \ft \cdot e_X =\zeta e_X\}$  is a basis for the corresponding eigenspace of $\rho(\ft)$. Let $\zeta \in \spec{\rho(\ft)|_W}$ and $w \in W\setminus\{0\}$ such that $\ft \cdot w=\zeta w$. Then 
 $w$ is a $\BC$-linear combination of $B_\zeta$. Thus,  for any $\s \in \GQ$, we have $^{\s^2}\!\ft \cdot w = \s^2(\zeta) w$ and $\s^{-1}\cdot w \in W$ by \eqref{eq: g_sigma}. It follows from \eqref{eq:GQ-SL-relation} that
 $$
 \ft \cdot (\s^{-1} \cdot w) = \s^{-1} \cdot({^{\s^2}\!\ft}\cdot w) = \s^2(\zeta)\  \s^{-1} \cdot w,
 $$
 and so $\s^2(\zeta) \in \spec{\rho(\ft)|_W}$.
\end{proof}

The minimal possible dimension of 
an 
$\SLZ$-submodule of $V_\CC$ of the preceding proposition inspires the following definition. 

\begin{defn}
A level $m$ representation $(\phi, W)$ of $\SLZ$  is called \emph{minimal}  if $\dim(W) = \varphi_2(m)$ and $$
\spec{\phi(\ft)}=\{\s^2(\zeta_m^l)\mid \s \in \gal(\BQ_m/\BQ)\}
$$
for some $l \in\UZ{m}$. In this case, $(\phi, W)$  or the corresponding $\SLZ$-module is said to be \emph{minimal of type $l$}.
\end{defn}

\begin{cor}\label{c:minimal_is_irr}
Let $(\rho, V_\CC)$ be a representation of $\SLZ$ associated with a modular category $\CC$. If $(\rho|_W, W)$ is a minimal subrepresentation of $(\rho, V_\CC)$, then $(\rho|_W, W)$ is irreducible. 
\end{cor}
\begin{proof}
Since $(\rho, V_\CC)$ is of some level $n=\ord(t)$,  $\ker (\rho|_W)$ is  a congruence subgroup of $\SLZ$. Let $m$ be the level of $(\rho|_W, W)$. Since $(\rho|_W, W)$ is minimal, $\dim(W) = \varphi_2(m)$ and 
$$
\spec{\rho(\ft)|_W}=\{\s^2(\zeta_m^l)\mid \s \in \gal(\BQ_m/\BQ)\} \quad\text{for some }l \in\UZ{m}\,.
$$
In particular, $\GQ$ acts transitively on $\spec{\rho(\ft)|_W}$. If $(\rho|_U, U)$ is a nontrivial subrepresentation of $(\rho|_W, W)$ and $\zeta \in \spec{\rho(\ft)|_U}$, then the $\GQ$-orbit of $\zeta$ is $\spec{\rho(\ft)|_W}$. Therefore, 
$$
\dim(U) \ge |\spec{\rho(\ft)|_W}| = \varphi_2(m)=\dim(W)\,.
$$
Therefore, $U = W$ and hence $(\rho|_W, W)$ is irreducible. 
\end{proof}

The following examples are building blocks of all the minimal irreducible representations of $\SLZ$.

\begin{example}\label{rmk:eta-12}
For any odd prime $p$, there are precisely two inequivalent irreducible representations of $\SLZ$ of level $p$ and dimension $\varphi_2(p)=(p-1)/2$, denoted by $(\eta^p_j, \BC^{\varphi_2(p)})$ or simply $\eta^p_j$  ($j=\pm 1$), which can be described as follows (see, for example,  \cite[Sec.~4]{eh93}). Let  $a \in (\Zn{p})^\times$, and  set $j=\jacobi{a}{p}$, the Legendre symbol of $a$ modulo $p$. For any integers $x,y \in [1, (p-1)/2]$,  
\begin{equation}\label{eq:eta}
\eta^p_j(\fs)_{x, y} = \frac{2i\,j}{\sqrt{p^*}}\sin\left(\frac{4\pi\, a xy}{p}\right)\,
\text{ and }\,
\eta^p_j(\ft)_{x, y} = \delta_{x,y}\,\exp\left(\frac{2\pi i\, a x^2}{p}\right)\,
\end{equation}
where
$$
\sqrt{p^*} = \left\{\begin{array}{cl}
\sqrt{p} &\text{if }\ p \equiv 1 \pmod 4\,,\\
-i\sqrt{p} &\text{if }\ p \equiv 3 \pmod 4\,.
\end{array}\right.
$$
The representation type of $\eta^p_j$ is independent of the choice of $a$ with $\jacobi{a}{p}=j$. The standard basis for $\BC^{\varphi_2(p)}$ is an eigenbasis of $\eta^p_j(\ft)$ and the representation $\eta^p_j$ is uniquely determined by $\spec{\eta^p_j(\ft)}$, which is either $\{\s^2(\zeta_p)\mid \s \in \gal(\BQ_p/\BQ)\}$ or  $\{\s^2(\zeta^a_p)\mid \s \in \gal(\BQ_p/\BQ)\}$ where $a$ is quadratic nonresidue modulo $p$. In particular, $\eta^p_{\pm 1}$ are level $p$ minimal representations of $\SLZ$.
\end{example}

\begin{example}\label{rmk:chi}
The isomorphism classes of 1-dimensional representations of $\SLZ$ form a cyclic group of order 12 under tensor product, and they are completely determined by the images of $\ft$. If $x$ is a 12-th root of unity,  we denote by $\chi_x$ the 1-dimensional representation of $\SLZ$ such that $\chi_x(\ft)=x$. In particular, $\chi^{\pm 1}_{\zeta_3}=\chi_{\zeta^{\pm 1}_3} = \eta^3_{\pm 1}$, and the level of $\chi_x$ is the order of $x$. Since $\ord(x) \mid 12$ and $\varphi_2(d)=1$ for any positive integer $d \mid 12$, every 1-dimensional representation of $\SLZ$ is minimal.
\end{example}

We close this subsection with the following characterization of minimal irreducible representations of $\SL_2(\BZ)$ which extends the preceding examples 
to 
a general setting. 

\begin{lem} \label{l:unique_minimal}
Let   $(\phi, V)$ be a level $n$ irreducible representation of $\SLZ$.  If $(\phi, V)$ is  minimal of type $l$, 
then $n=d\cdot  p_1 \cdots p_\ell$ for some positive integer $d \mid 12$ and distinct primes $p_1, \dots, p_\ell \ge 5$. In this case, there exist unique $l_0 \in (\Zn{d})^\times$ and $l_i \in (\Zn{p_i})^\times$ such that 
$\zeta_n^l = \zeta_d^{l_0} \zeta_{p_1}^{l_1}\cdots \zeta_{p_\ell}^{l_\ell}$ and
$$
\phi \cong \chi_{x}\ \o\, \eta_{j_1}^{p_1} \o \cdots \o\, \eta_{j_\ell}^{p_\ell},
$$
where $x=\zeta^{l_0}_{d}$ and $j_i = \jacobi{l_i}{p_i}$. In particular, $\phi$ is uniquely determined by $\zeta_n^l$ up to equivalence.
\end{lem}
\begin{proof}
Let $p$ be a prime factor of $n$, and $m$ a positive integer such that $n=p^m \cdot n_2$, where $n_2$ is a positive integer not divisible by $p$. Set $n_1 =p^m$.
By the Chinese Remainder Theorem, there exist irreducible representations $\phi_i: \SLZ \to \GL(V_i)$ of level $n_i$  such that 
$$
\phi \cong  \phi_1 \o \phi_2\,.
$$
Therefore, for any $\w\in \spec{\phi(\ft)}$, 
$$
\w = \w_1 \cdot \w_2
$$
where $\w_i \in \spec{\phi_i(\ft)}$. Since $(\phi, V)$ is minimal of type $l$, $\w=\s^2(\zeta_n^l)$ for some $\s \in \gal(\BQ_n/\BQ)$, which means it is a primitive $n$-th root of unity. Thus, $\w_i$ is primitive $n_i$-th root for $i=1,2$. Note that the group $\mu_n$ of $n$-th roots of unity is an internal direct product of $\mu_{n_1}$ and $\mu_{n_2}$, the pair $(w_1, \w_2)$ is uniquely determined by $\w$. More precisely, there exists a unique $l_i \in (\Zn{n_i})^\times$ such that  $l = l_i n/n_i$ in $\Zn{n_i}$. Then
$$
\zeta_n^l = \zeta_{n_1}^{l_1}\cdot \zeta_{n_2}^{l_2}
$$
and
$$
\w_i = \s^2\left(\zeta_{n_i}^{l_i}\right)
$$
for $i=1, 2$.  As $\s$ runs through $\gal(\BQ_n/\BQ)$, we find
$$
\{\s^2\left(\zeta^{l_i}_{n_i}\right)\big|\ \s \in \gal(\BQ_{n_i}/\BQ)\} 
$$
is a subset of $\spec{\phi_i(\ft)}$. Therefore, $\dim(V_i) \ge \varphi_2(n_i)$ and so
$$
\varphi_2(n) \ge \dim(V_1)\cdot \dim(V_2) \ge \varphi_2(n_1) \cdot \varphi_2(n_2)  = \varphi_2(n)\,.
$$
This implies $\dim(V_i) = \varphi_2(n_i)$ and  $$\spec{\phi_i(\ft)}=\{\s^2\left(\zeta^{l_i}_{n_i}\right)\big|\ \s \in \gal(\BQ_{n_i}/\BQ)\}\,.
$$
Thus, both $\phi_1$ and $\phi_2$ are minimal of type $l_1$ and $l_2$ respectively.

The level $p^m$ irreducible representations of $\SLZ$ were classified by \cite{Nobs1, NW76} (see also \cite[Tbl.~1 - 8]{Eh}). Since $\phi_1$ is an irreducible representation of level $p^m$ and dimension $\varphi_2(p^m)$, whose values are given by \eqref{eq:phi2}, we find
$$
m=\left\{\begin{array}{ll}
1  & \text{if  $p$ is odd}; \\
1 \text{ or } 2 & \text{if  $p=2$}.
\end{array}\right.
$$
In this case, $\phi_1 \cong \eta^p_{\pm 1}$ 
if  
$p >3$ (cf.~Remark \ref{rmk:eta-12}) and $\phi_1$ is 1-dimensional if $p \le 3$. Since $p$ can be any prime factor of $n$, we obtain the factorization $n=d\cdot  p_1 \cdots p_\ell$ for some positive integer $d \mid 12$ and $p_1$, ..., 
$p_{\ell}$ are distinct primes greater than 3.

If one denotes the preceding irreducible representation $\phi_1$ by $\phi^p$, then, by induction, we have
$$
\phi \cong \phi^d \ \ \o \ \bigotimes_{\substack{\text{prime }p > 3 \\ p \mid n}} \phi^p,
\quad
\text{where} \quad
\phi^d = \bigotimes_{\substack{\text{prime }p \le 3 \\ p \mid n}} \phi^p
$$
is 1-dimensional. There exist a unique integer $l_p\pmod{p}$ satisfying $l \equiv l_p n/p \pmod{p}$ for each odd prime divisor $p$ of $n$, and a unique $l_0 \in (\Zn{d})^\times$ satisfying $l \equiv l_0 n/d \pmod{d}$. Then, we have
$$
\zeta_n^l = \zeta_d^{l_{0}} \prod_{\substack{\text{prime }p > 3\\ p \mid n}}\zeta_p^{l_p} \quad\text{and}\quad \zeta_d^{l_0} = \phi^d(\ft)\,.
$$
Therefore, $\phi^d = \chi_{\zeta_d^{l_0}}$ and $\phi^p=\eta^p_{j_p}$, where $j_p = \jacobi{l_p}{p}$ (cf. Examples \ref{rmk:eta-12} and \ref{rmk:chi}). Consequently,
\begin{equation*}
\phi \cong \chi_{\zeta_d^{l_0}}\, \o \bigotimes_{\substack{\text{prime }p > 3 \\ p \mid n}} \eta_{j_p}^p\,. 
\qedhere
\end{equation*}
\end{proof}

\subsection{Characteristic 2-group of modular categories}
Let $\CC$ be a modular category with the modular data $(S,T)$. For any normalized modular data $(s,t)$ of $\CC$,  $\BQ(S) \subset \BQ_N \subseteq \BQ_n$, where $N=\ord(T)$ and $n =\ord(t)$ (cf.~\cite{NS10, DongLinNg}). The restriction of the Galois automorphisms of $\BQ_n$ to $\BQ(S)$ defines an epimorphism $\res^{\BQ_n}_{\BQ(S)}: \gal(\BQ_n/\BQ) \to G_\CC$ of groups. Note that 
by \cite[Prop.~6.7]{DongLinNg}, we have
\begin{equation}\label{eq:ker_res}
   \ker(\res^{\BQ_n}_{\BQ(S)})=\gal(\BQ_n/\BQ(S)) \subseteq \O2\,.
\end{equation}

\begin{defn}\label{def:HC}
Let $(s,t)$ be a normalized modular data of a modular category $\CC$, and $n=\ord(t)$.  The image of the elementary 2-subgroup $\O2$ of $\gal(\BQ_n/\BQ)$ under the  restriction map $\res^{\BQ_n}_{\BQ(S)}: \gal(\BQ_n/\BQ) \to G_\CC$ is  called the \emph{characteristic 2-group} of $\CC$, and denoted by $H_\CC$.
\end{defn}

In view of \eqref{eq:ker_res}, we have the exact sequence of abelian groups:
\begin{equation}\label{eq:exact_HC}
   1 \rightarrow \gal(\BQ_n/\BQ(S)) \xrightarrow{incl} \O2 \xrightarrow{\res_{\BQ(S)}^{\BQ_n}} H_\CC \rightarrow 1\,.
\end{equation}

\begin{prop}\label{l:phi2}
The characteristic 2-group $H_\CC$ of $\CC$ is independent of the choice of  the normalized modular data $(s,t)$ of $\CC$. Moreover, if $n=\ord(t)$, then
$$
G_\CC/H_\CC\cong  \frac{\gal(\BQ_n/\BQ)}{\O2}\,.
$$
In particular,
$|G_\CC|/ |H_\CC|= \varphi_2(n)$.
\end{prop}
\begin{proof}
Let $(s,t)$ and $(s',t')$ be normalized modular data of $\CC$ and let $(\rho, V_\CC)$ and $(\rho', V_\CC)$ be the corresponding representations of $\SLZ$ associated with $\CC$ respectively. Then $\rho' \cong \chi \o \rho$ for some 1-dimensional character of $\SLZ$. Since $\chi^{12}=1$, $t' = x t$ for some 12-th root 
of unity $x$. 
Let $m = \ord(t')$, and $l = \lcm(m,n)$. Then $\BQ_l=\BQ_m(x) = \BQ_n(x)$. By definition, 
$\res^{\BQ_l}_{\BQ_n} (\Omega_2^{l}) \subseteq \O2$. For any $\s \in \O2$, there exists 
an extension $\tau \in \gal(\BQ_{l}/\BQ)$ such that $\tau|_{\BQ_n} = \s$. Since $x^{12}=1$, ${\tau}^2(x)=x$. Thus, $\tau^2 = \id$ and hence $\tau \in \Omega_2^{l}$. Therefore, 
$$
\res^{\BQ_l}_{\BQ_n} (\Omega_2^{l}) = \O2\,.
$$
By the same argument, we also have 
$$
\res^{\BQ_l}_{\BQ_m} (\Omega_2^{l}) = \Omega_2^{m}\,.
$$
Since the diagram
$$
\xymatrix{
\gal(\BQ_l/\BQ) \ar[rr]^-{\res_{\BQ_m}^{\BQ_{l}}} \ar[rrd]^-{\res_{\BQ(S)}^{\BQ_l}} \ar[d]_-{\res_{\BQ_n}^{\BQ_{l}}}
&& \gal(\BQ_m/\BQ) \ar[d]^-{\res_{\BQ(S)}^{\BQ_{m}}}\\
\gal(\BQ_n/\BQ) \ar[rr]_-{\res_{\BQ(S)}^{\BQ_{n}}} && G_\CC
}\,
$$
of restriction maps is commutative, we have
$$
\res_{\BQ(S)}^{\BQ_m} (\Omega_2^m) = \res_{\BQ(S)}^{\BQ_l} (\Omega_2^{l})= \res_{\BQ(S)}^{\BQ_n}(\O2)\,. 
$$
This proves the first assertion of the statement.

By \eqref{eq:exact_HC}, we also have the following commutative diagram of abelian groups with exact rows:
$$
\xymatrix{
1 \ar[r] & \gal(\BQ_n/\BQ(S)) \ar[r]^-{incl} \ar[d]_-{\id} & \O2  \ar[r]^-{\res^{\BQ_n}_{\BQ(S)}} \ar[d]_-{incl}  & H_\CC  \ar[r] \ar[d]^-{incl} & 1\\
1 \ar[r] & \gal(\BQ_n/\BQ(S)) \ar[r]^-{incl}  & \gal(\BQ_n/\BQ) \ar[r]^-{\res^{\BQ_n}_{\BQ(S)}} & G_\CC  \ar[r] & 1\,.\\
}
$$
Therefore, 
\begin{equation*}
G_\CC/H_\CC \cong \frac{\gal(\BQ_n/\BQ)/\gal(\BQ_n/\BQ(S))}{\O2/\gal(\BQ_n/\BQ(S))} \cong  \frac{\gal(\BQ_n/\BQ)}{\O2}\,. 
\qedhere
\end{equation*}
\end{proof}

\begin{cor}\label{c:HC}
Let $\CC$ be a modular category with the modular data $(S,T)$. If $N= \ord(T)$ is not a multiple of $4$, then 
$$
H_\CC = \res_{\BQ(S)}^{\BQ_N} (\Omega_2^N)\,. 
$$
\end{cor}
\begin{proof}
Since $4 \nmid N$, by \cite[Lem. 2.2]{DongLinNg}, there exists a level $N$ representation $(\rho, V_\CC)$ of $\SLZ$ associated with $\CC$. Therefore, $\rho(\ft)=t$ has order $N$. Now, the result follows directly from Definition \ref{def:HC}  of $H_\CC$.
\end{proof}

\begin{example}  Let $A$ be a finite abelian group  and $q: A \to \BC^\times$ a nondegenerate quadratic form. The pointed modular category $\CC = \CC(A, q)$ has the S- and T-matrices given by
$$
S_{a,b} = \frac{q(a)q(b)}{q(ab)}, \quad T_{a,b} = \delta_{a,b} q(a)
$$
for any $a, b \in A$. 
\begin{enumerate}
    \item[(i)] If $|A|$ is odd, then $\BQ(S) = \BQ(T)=\BQ_N$, where $N =\ord(T)$. Since $|A|$ is odd, and so 
    is 
    $N$. Therefore, by Corollary \ref{c:HC}, $H_\CC = \ON$ is nontrivial.
    \item[(ii)] If $A=\langle a \rangle$ is a cyclic group of order 2 and and $q(a)=\pm i$, then $\CC$ is called a \emph{semion category}. In this case, $\ord(T)=4$ and $\BQ(S)=\BQ$. Therefore, $H_\CC$ is trivial. 
\end{enumerate}
\end{example}

Let  $(\rho, V_\CC)$ be a level $n$ representation of  $\SLZ$ associated with $\CC$, and $(s,t)$ the corresponding normalized modular data. Since $\O2 \xrightarrow{\res} H_\CC$ is an epimorphism of elementary 2-groups, there exists a subgroup $\tH\subset \O2$ such that  
\begin{equation}\label{eq:HC-tHC}
\res^{\BQ_n}_{\BQ(S)}: \tH \xrightarrow{\sim} H_\CC
\end{equation}
is an isomorphism. Now, recall that the $\GQ$-action  on $V_\CC$ via $\phi_\rho$  factors through $\gal(\BQ_n/\BQ)$ (cf.~\eqref{eq:phi-rho}). 
Therefore, $\tH$ acts on $V_\CC$ in the same way (cf.~\eqref{eq:GQ-action}), namely
$$
\s \cdot e_X = g_\s(e_X) = \e_\s(X) e_{\hs(X)}
$$
for any $\s \in \tH$.
One can decompose 
$V_{\CC}$ as an $\tH$-module
into its isotypic components
$$
V_\CC = \bigoplus_{\chi\in \irr(\tH)} V_\CC^\chi\,,
$$
where $\irr(\tH)$ 
denotes
the set of irreducible characters of $\tH$, and $V_\CC^\chi$ the isotypic component 
of 
$V_\CC$ corresponding to the irreducible character $\chi$ of $\tH$.

\begin{prop} \label{p:HC-decomp} 
Let $\CC$ be a modular category and $(\rho, V_\CC)$ a representation of $\SLZ$  associated with $\CC$. Then for any $\chi \in \irr(\tH)$, the isotypic component $V_\CC^\chi$ is an $\SL_2(\BZ)$-submodules of $V_\CC$, and 
\begin{equation}\label{eq:HC-decomp}
    V_\CC = \bigoplus_{\chi\in \irr(\tH)} V_\CC^\chi
\end{equation}
is a decomposition of $\SL_2(\BZ)$-modules.
Moreover, if there exists a simple object $X \in \irr(\CC)$ such that $\stab_{H_{\CC}}(X) = \{\id\}$, then all the $V_\CC^\chi$'s are non-zero and pairwise inequivalent.
\end{prop}

\begin{proof} By \eqref{eq:GQ-SL-relation}, for any $v\in V_\CC^\chi$, $\s \in  \tH$, and $\fa\in \SL_2(\BZ)$, 
$$
\s\cdot (\fa \cdot v) =\ {^{\s^2}}\!\fa\cdot  (\s\cdot v) = \chi(\s)\ \fa\cdot v\,.
$$ 
Therefore, $V_\CC^\chi$ is an $\SLZ$-invariant subspace of $(\rho, V_{\CC})$, and the  $\SL_2(\BZ)$-module decomposition \eqref{eq:HC-decomp} follows immediately. 

Let $\chi, \chi'$ be distinct irreducible characters of $\tH$ such that $V_\CC^\chi \ne 0$ and $V_\CC^{\chi'}\ne 0$. Then, there exists $\s \in \tH$ such that $\chi(\s) \ne \chi'(\s)$. By \eqref{eq: g_sigma}, $g_\s = \rho(\fa)$ for some $\fa \in \SL_2(\BZ)$, and the restrictions of $\rho(\fa)$ on $V_\CC^\chi$ and $V_\CC^{\chi'}$ are the distinct scalars $\chi(\s)$ and $\chi'(\s)$ respectively. Therefore, $V_\CC^\chi$ and $V_\CC^{\chi'}$ are inequivalent representations of $\SL_2(\BZ)$. 

For each $\chi \in \irr(\tH)$, 
$$
P_\chi:=\frac{1}{|\tH|} \sum_{\s \in \tH} \chi(\s) g_\s
$$
is an idempotent operator on $V_\CC$ commuting with the action $\SLZ$ such that $V_\CC^\chi = P_\chi (V_\CC)$. Therefore, $V_\CC^\chi=0$ if and only if $P_\chi =0$. If $\{g_\s\mid \s \in \tH\}$ is $\BC$-linearly independent, then $P_\chi \ne 0$ and hence $V_\CC^\chi \ne 0$ for all $\chi \in \irr(\tH)$.

Let $X \in \irr(\CC)$ be such that $\stab_{H_{\CC}}(X) = \{\id\}$. Suppose $\sum_{\s \in \tH} \a_\s g_\s = 0$ for some $\a_\s \in \BC$. Then
	$$
	\sum_{\s \in \tH} \a_\s g_\s(e_X) = \sum_{\s \in \tH} \a_\s \e_\s(X) e_{\hs X} = 0\,.
	$$
	Since $\stab_{H_{\CC}}(X) = \{\id\}$, $\{e_{\hs X} \mid \hs \in H_\CC\}$ is a set of distinct basis elements of $V_{\CC}$ and hence $\a_\s = 0$ for all $\s \in \tH$. Therefore, $\{g_\s\mid\s \in \tH\}$ is $\BC$-linearly independent, and so $V_\CC^\chi \ne 0$ for all $\chi \in \tH$.	 This completes the proof of the proposition. 
\end{proof}

\begin{prop} 
Let  $(\rho, V_\CC)$ be a level $n$ representation of $\SLZ$ associated with a modular category $\CC$. If $\rho$ is irreducible, then $H_\CC$ is trivial, the S-matrix of $\CC$ is real, and  $\CC$ is self-dual.  Moreover, there exists $X \in \irr(\CC)$ such that $\rho(\ft)_{X,X}$ is a primitive  $n$-th root of unity.
\end{prop}
\begin{proof}
If $\rho$ is an irreducible representation of $\SL_2(\BZ)$, the decomposition \eqref{eq:HC-decomp} of $\rho$, determined by the characteristic 2-subgroup $H_\CC$, must be trivial. Suppose there exists a nontrivial  element  $\hs$ in $H_\CC$. Then $\hs(X) \ne X$ for some $X \in \irr(\CC)$ and so  the eigenspaces $E_\pm$ of $g_\s$ corresponding to the eigenvalues $\pm 1$ are nontrivial. Note that both $E_+$ and $E_-$  are stable under the $\SLZ$ action, and  $V_\CC = E_+ \oplus E_-$. This contradicts the irreducibility of $V_\CC$. Therefore, $H_\CC$ is trivial.

Let $\s \in \gal(\BQ_n/\BQ)$ denote the complex conjugation. Then $\hs(X)=X^*$ for $X \in \irr(\CC)$. Since $H_\CC$ is trivial, $\s|_{\BQ(S)} =\id$ and so $X^*= \hs(X)=X$ for $X \in \irr(\CC)$. Therefore, $S_\CC$ is real  and $\CC$ is self-dual.

Let $n=p_1^{n_1}\cdots p_\ell^{n_\ell}$ be the prime factorization of $n$, where $p_1, \dots, p_\ell$ are distinct prime factors of $n$.  Since $\rho$ is irreducible, by the Chinese Remainder Theorem, 
there exists a  level $p_i^{n_i}$ irreducible representation $(\rho_i,V_i)$ of $\SLZ$ for each $i=1,\dots, \ell$  such that
$$
(\rho, V_\CC) \cong  (\rho_1, V_1) \o \cdots \o (\rho_\ell, V_\ell)\,.
$$
Since $(\rho_i, V_i)$ is of level $p_i^{n_i}$, there exists a nonzero eigenvector $v_i \in V_i$ of $\rho_i(\ft)$ with an eigenvalue $\w_i$ which is a primitive $p_i^{n_i}$-th root of unity. Thus, $\rho(\ft)$ has an eigenvalue $\zeta=\w_1\cdots\w_\ell$ which is a primitive $n$-th root of unity. Since 
$\{e_X\mid X \in \irr(\CC)\}$
is an eigenbasis for $\rho(\ft)$, there exists $X \in \irr(\CC)$ such that $\rho(\ft)_{X,X}=\zeta$. 
\end{proof}

\subsection{The \texorpdfstring{$\SLZ$}{}-modules of transitive modular categories} In this section, we show that the representations of $\SLZ$ associated with 
any 
transitive modular 
category
$\CC$ 
is minimal and irreducible, and that the order of $T_\CC$ is odd and square-free. 

Let $\CC$ be a transitive modular category, $(\rho,V_\CC)$ a level $n$ representation of $\SLZ$ associated with $\CC$, and  $(s, t)$ the corresponding normalized modular data. As before, the Galois group $G_\CC$ is identified with $\irr(\CC)$ via the bijection $\hs \mapsto \hs(\1)$.  Then, we have  
\begin{equation}\label{eq:spec_t}
    \spec{t} = \{t_{\hs, \hs}\mid  \hs\in G_\CC\} = \{\s^2(\zeta)\mid \s \in\gal(\BQ_n/\BQ)\},
\end{equation}
where $\zeta=t_{\1, \1}$. Here, the last equality is a consequence of \eqref{eq:gal_t}.  Therefore, $\GQ$ acts on $\spec{t}$ transitively, and so every eigenvalue of $t$ is a primitive $n$-th root of unity. In particular, 
$$
\BQ(t) = \BQ_n = \BQ(\zeta)\,.
$$
\begin{lem}\label{lem:HC} The characteristic 2-group $H_\CC$ is given by
$$
H_\CC = \{ \hs \in G_\CC \mid t_{\hs, \hs} = t_{\1,\1} \}\,.
$$
Moreover, for any $\hs, \hta \in H_\CC$, $t_{\hs, \hs}= t_{\hta, \hta }$ if and only if $\hs H_\CC = \hta H_\CC$. In particular, each eigenvalue of $t$ has algebraic multiplicity $|H_\CC|$.
\end{lem}
\begin{proof}
Let $\zeta= t_{\1, \1}$.  Since  $\BQ(\zeta)=\BQ_n$,   we have 
$$
\O2=\{\s \in \gal(\BQ_{n}/\BQ) \mid \s^2(\zeta)= \zeta\} = \{\s \in \gal(\BQ_{n}/\BQ) \mid t_{\hs, \hs}= \zeta\}.
$$
Thus,  if $\hs \in H_\CC$, then there exists $\s \in \O2$ such that  $\s|_{\BQ(S)} = \hs$, which means $t_{\hs, \hs}=\zeta$. Conversely, if  $\hs \in G_\CC$ such that $t_{\hs, \hs}=\zeta$, then there exists $\s \in \gal(\BQ_n/\BQ)$ such that  $\s|_{\BQ(S)} = \hs$. By \eqref{eq:gal_t}, 
$\s^2(\zeta) =t_{\hs, \hs}=\zeta$. 
Thus, $\s \in \O2$, and hence $\hs \in H_\CC$. This proves the first statement.

Let $\hs, \hta \in G_\CC$, and $\tau \in \gal(\BQ_n/\BQ)$ such that $\tau|_{\BQ(S)} = \hta$. If $t_{\hs, \hs} = t_{\hta, \hta}$, then 
$t_{\hs, \hs} = \tau^2(\zeta)$ or 
$$
\zeta = \tau^{-2}(t_{\hs, \hs} ) = t_{\hta^{-1}\hs, \hta^{-1}\hs}\,.
$$
Therefore, $\hta^{-1} \hs \in H_\CC$ and so  $\hta H_\CC = \hs H_\CC$. Conversely, if $\hta H_\CC = \hs H_\CC$, then $\hs =\hta \hm $ for some $\hm \in H_\CC$, and hence
\begin{equation*}
t_{\hs,\hs} = t_{\hta \hm,\hta \hm} = \tau^2(t_{\hm,\hm}) =  \tau^2(t_{\1,\1}) =  t_{\hta,\hta}\,.
\qedhere
\end{equation*}
\end{proof}
Now, we can prove the major theorem of this section. 
\begin{thm}\label{thm:tran_irr}
Let $\CC$ be a nontrivial transitive modular category. Then every representation of $\SLZ$ associated with $\CC$ is  minimal and irreducible.  Moreover, the order of the T-matrix $T$ of $\CC$ is odd and square-free, and every prime factor of $\ord(T)$ is greater than 3.
\end{thm}
\begin{proof}
Let $(\rho, V_\CC)$ be a level $n$ representation of $\SLZ$ associated with $\CC$, and $H_\CC$ the characteristic 2-group of $\CC$.
By Propositions \ref{p:sizeofGalois} and \ref{p:HC-decomp}, with $\tH$ defined in \eqref{eq:HC-tHC},  $V_\CC$ admits an $\SLZ$-module decomposition 
$$
V_\CC = \bigoplus_{\chi \in \irr(\tH)} V_\CC^\chi
$$
such that $V_\CC^\chi \ne 0$ for all $\chi \in \tH$. We proceed to determine $V_\CC^\chi$ for each $\chi  \in \irr(\tH)$.

Recall that the $\tH$-action on $V_\CC$ is given by
$$
\s \cdot e_\hm = g_\s (e_\hm) = \e_\s(\hm) e_{\hs\hm}
$$
for any $\s \in \tilde H_\CC$ and $\hm \in G_\CC$. For any $\hm \in G_\CC$, the subspace $V_{\hm}$ of $V_\CC$ spanned by $\{e_{\hs \hm}\mid \hs \in H_\CC\}$ is closed under this $\tilde H_\CC$-action, and $\ft$ acts as the scalar $t_{\hm, \hm}$ on $V_{\hm}$ by Lemma \ref{lem:HC}. 
Therefore, $V_{\hm}$ admits an isotypic decomposition 
$$
V_{\hm} = \bigoplus_{\chi \in \irr(\tH)} V_{\hm}^\chi \,.
$$

Since $g_{\id} = \id_{V_\CC}$, the character $\psi_\hm$ of $\tH$ afforded by $V_\hm$ is given by
$$
\psi_\hm(\s) = |H_\CC|\cdot \delta_{\s, \id} \quad \text{for any } \s \in \tH\,.
$$
Therefore, as an $\tH$-module, $V_\hm$ is equivalent to the regular representation of $\tH$. Consequently, $\dim(V_\hm^\chi)=1$ for each $\chi \in \irr(\tH)$. 

Let $\Lambda$ be a complete set of coset representatives of $H_\CC$ in $G_\CC$.
Then, 
$$
V_\CC = \bigoplus_{\hm \in \Lambda} V_\hm
$$
is a decomposition of $\tH$-modules. Therefore, 
$$
V_\CC^\chi = \bigoplus_{\hm \in \Lambda} V_\hm^\chi
$$ for each $\chi \in \irr(\tH)$, and $\dim(V_\CC^\chi)=|G_\CC|/|H_\CC| = \varphi_2(n)$ by Proposition \ref{l:phi2}. Let $(\rho^\chi, V_\CC^\chi)$ denote the corresponding subrepresentation of $(\rho, V_\CC)$. Then
$$
\spec{\rho^\chi(\ft)} =\{t_{\hs,\hs}\mid \hs \in \Lambda\}=\{\s^2(t_{\1,\1})\mid \s \in \GQ\}
$$
by Lemma \ref{lem:HC}. Therefore, for any $\chi \in \irr(\tH)$, the level $n$ representation $(\rho^\chi, V_\CC^\chi)$ of $\SLZ$ is 
minimal of type $l \in \UZ{n}$, where $l$ is determined by $\zeta_n^l=t_{\1, \1} $.  Hence, by Corollary \ref{c:minimal_is_irr}, $(\rho^\chi, V_\CC^\chi)$ is irreducible for each $\chi \in \tH$.

It follows from Lemma \ref{l:unique_minimal} that $V_\CC^\chi \cong V_\CC^{\chi'}$ as $\SLZ$-modules for any $\chi, \chi' \in \irr(\tH)$. In view of Proposition \ref{p:HC-decomp}, $\tH$ must be trivial and so does $H_\CC$. Therefore, $(\rho, V_\CC)$ is minimal and irreducible, and $n = d\cdot p_1\cdots p_\ell$ where $d \mid 12$ and $p_1, \dots, p_\ell$ are distinct primes greater than 3. Moreover,
$$
\rho \cong \chi \o \rho'
$$
for some 1-dimensional representation $\chi$ and a level $m=p_1\cdots p_\ell$ 
minimal 
representation $(\rho', V')$ of $\SLZ$. By tensoring $\rho$ with the dual representation $\chi^*$ of $\chi$, we find $(\rho', V')$ is equivalent to a representation of $\SLZ$ associated with $\CC$. By \cite[Thm II (i)]{DongLinNg}, we have
$$
\ord(T) \mid m \mid 12 \ord(T)
$$
which implies  $\ord(T) = m$ since $\gcd(m,12)=1$.
\end{proof}

\section{Classification of transitive modular categories}\label{sec:class}
In this section, we prove that a nontrivial prime and transitive modular category must be equivalent to $\BB_{p-2, l}$ for some prime $p > 3$ and $l \in \UZ{2p}$. In view of Theorem \ref{thm:prime-decomp}, we complete the classification of transitive modular categories in Theorem \ref{thm:main}. The 
minimal irreducibility of the representations of $\SLZ$  associated with transitive modular categories is  crucial to the characterization of the prime ones. 

We begin with the realization of minimal irreducible representations of $\SLZ$ by transitive modular categories.

\begin{lem}\label{l:realizing_min}
Let $p >3$ be a prime. Then every level $p$ minimal irreducible representation of $\SLZ$ is equivalent to a representation of $\SLZ$ associated with $\BB_{p-2, l}$ for some $l \in \UZ{2p}$. 
\end{lem}
\begin{proof}
Recall from Proposition \ref{p:B-trans-prime} that  $\BB_{p-2,l}$ is a prime and transitive modular category  for 
any prime $p >3$ and $l \in \UZ{2p}$. Moreover, the order of the T-matrix of $\BB_{p-2,l}$ is $p$. 
By \cite[Lemma 2.2]{DongLinNg}, 
there exists a level $p$ representation $(\rho, \BC^{\varphi_2(p)})$ of $\SLZ$ associated with $\BB_{p-2,1}$, and we set $t=\rho(\ft)$. Then, by Theorem \ref{thm:tran_irr}, $(\rho, \BC^{\varphi_2(p)})$ is a minimal  irreducible representation of $\SLZ$ of type $a$ where $\zeta_p^a = t_{\1, \1}$.  Therefore,  by Lemma \ref{l:unique_minimal},
$$
(\rho, \BC^{\varphi_2(p)}) \cong (\eta^p_j, \BC^{\varphi_2(p)}),\quad  \text{where }j=\jacobi{a}{p}\,.
$$

For any $l \in \UZ{2p}$, define $\s_l \in \gal(\BQ_p/\BQ)$ 
by $\s_l(\zeta_p) = \zeta_p^l$. Since $\rho(\fa)$ is a matrix over $\BQ_p$ for any $\fa \in \SLZ$ (cf.~\cite[Thm.~II]{DongLinNg}), $\rho_l(\fa) := \s_l(\rho(\fa))$ 
defines another level $p$ representation of $\SLZ$, and $(\rho_l, \BC^{\varphi_2(p)})$ is a representation of $\SLZ$ associated with $\BB_{p-2,l}$. Since $\s_l(t_{\1, \1}) = \zeta_p^{al}$, we have
$$
(\rho_l, \BC^{\varphi_2(p)}) \cong (\eta^p_{j_l}, \BC^{\varphi_2(p)}),\quad  \text{where }j_l=\jacobi{al}{p}\,.
$$
Therefore, every  level $p$ minimal irreducible representation of $\SLZ$ is equivalent to a representation of $\SLZ$ associated with $\BB_{p-2, l}$ for some $l \in \UZ{2p}$, as desired.
\end{proof}
\begin{cor}\label{c:realizing_min}
Let $n=p_1\cdots p_\ell$ for some distinct primes $p_1,\cdots, p_\ell > 3$. Then every level $n$ minimal irreducible representation of $\SLZ$ is equivalent to a representation of $\SLZ$ associated to a transitive modular category 
$$
\DD = \BB_{p_1-2, l_1} \bot \cdots \bot \BB_{p_\ell-2, l_\ell}
$$
for some  $l_a \in \UZ{2p_a}$, $a=1, \dots, \ell$.
\end{cor}
\begin{proof}
Let $(\phi, V)$ be a level $n$ minimal irreducible representation of $\SLZ$. By Lemma \ref{l:unique_minimal}, there exists level $p_a$ minimal irreducible representation $(\eta^{p_a}_{j_a}, V_a)$ of $\SLZ$ for each $a = 1, \dots, \ell$ such that 
$$
(\phi,V)\cong (\eta^{p_1}_{j_1}, V_1)\o \cdots \o (\eta^{p_\ell}_{j_\ell}, V_\ell)
$$
where $V_a =\BC^{\varphi_2(p_a)}$.  By Lemma \ref{l:realizing_min}, $(\eta^{p_a}_{j_a}, V_a)$ is equivalent to a representation $(\rho_a, V_{\DD_a})$ of $\SLZ$ associated with a transitive modular category $\DD_a=\BB_{p_a-2, l_a}$ for some $l_a \in \UZ{2p_a}$. Let $\DD = \DD_1 \bot  \cdots \bot \DD_\ell$. Then $\DD$ is transitive by Proposition \ref{p:prod_A} and
$$
(\rho, V_\DD) = (\rho_1, V_{\DD_1}) \o \cdots \o (\rho_\ell, V_{\DD_\ell}) 
$$
is a representation of $\SLZ$ associated with $\DD$. Now, we have
\begin{equation*}
(\phi,V)\cong (\rho, V_\DD)\,. 
\qedhere
\end{equation*}
\end{proof}

\begin{thm}\label{thm:ordT-prime}
Let $\CC$ be a nontrivial prime and transitive modular category.  Then the order of the T-matrix is a prime number greater than 3.
\end{thm}
\begin{proof}
By Theorem \ref{thm:tran_irr}, $\ord(T_\CC)=N$ is odd and has a prime factor $p >3$. It follows from \cite[Lem.~2.2]{DongLinNg} that there exists a level $N$ representation $(\rho, V_\CC)$ of $\SLZ$ associated with  $\CC$. Again, by Theorem \ref{thm:tran_irr}, $(\rho, V_\CC)$ is minimal and irreducible.

Suppose $N$ is not a prime. Then $N=pq$ for some odd square-free integer $q$ not divisible by $p$ and all the prime factors of $q$ are greater than 3. In particular, $\varphi_2(q) >1$. In view of Lemma \ref{l:unique_minimal}, there exist minimal and irreducible $\SLZ$-representations   $(\phi_1, V_1)$ and $(\phi_2, V_2)$ of levels $p$ and $q$ respectively such that 
\begin{equation}\label{eq:equiv1}
    (\rho, V_\CC) \cong (\phi_1, V_1) \otimes (\phi_2, V_2)\,.
\end{equation}
It follows from Lemma \ref{l:realizing_min} and Corollary \ref{c:realizing_min} that there exist modular categories $\B_1, \B_2$ such that $(\phi_i, V_i)$ is equivalent to a representation $(\rho_i, V_{\B_i})$ associated with $\B_i$ and 
$$
\B_1 = \BB_{p-2, l} \quad \text{for some } l \in \UZ{2p}\,.
$$
Note that $(\rho_1, V_{\B_1}) \o (\rho_2, V_{\B_2})$ is a representation of $\SLZ$ associated with $\B=\B_1 \bot\B_2$ and
\begin{equation}\label{eq:equiv2}
(\rho, V_\CC) \cong (\rho_1, V_{\B_1}) \o (\rho_2, V_{\B_2})\,.
\end{equation}
In particular, the eigenvalues of $\rho(\ft)$ are all distinct.

Let $E_i$ be the standard basis for $V_{\B_i}$.  Then, $E_i$ is an eigenbasis of $\rho_i(\ft)$ and 
$$E_\B=\{x_1\o x_2 \mid (x_1, x_2) \in E_1 \times E_2\}$$ 
is an eigenbasis of $\rho_1 (\ft) \o \rho_2(\ft)$ for $V_\B=V_{\B_1} \o V_{\B_2}$. Since $E_\CC=\{e_X \mid X \in \irr(\CC)\}$ is an eigenbasis of $\rho(\ft)=t$ for $V_\CC$, the equivalence \eqref{eq:equiv2} implies there exists a bijection $\Phi: \irr(\CC) \to E_1 \times E_2$, which is defined as follows: for any $X\in \irr(\CC)$, there exists a unique pair $(x_1, x_2) \in E_1 \times E_2$ satisfying
$$
(\rho_1(\ft) \o \rho_2(\ft)) (x_1\o x_2) = t_{X, X} \cdot x_1\o x_2\,,
$$
and we define $\Phi(X):=(x_1,x_2)$.

Let $\Phi(\1) = (b_1, b_2)$, and $D: = \Phi^{-1}(E_1 \times \{b_2\}) \subseteq \irr(\CC)$. Let $\DD$ be the full subcategory of $\CC$ additively generated by the simple objects  whose isomorphism classes are in $D$, i.e. $\DD$ is a semisimple subcategory of $\CC$ with $\irr(\DD) = D$.  We proceed to show $\DD$ is a fusion subcategory of  $\CC$.

By \cite[Lem.~3.17]{BNRWClassificationByRank}, there exists an intertwining operator $U: (\rho,V_\CC) \to (\rho_1 \o \rho_2, V_\B)$ such that for any $X \in \irr(\CC)$, 
$U(e_X) = U_{(x_1,x_2)}\ x_1\o x_2$ for some scalar $U_{(x_1,x_2)} =\pm 1$ where $\Phi(X)=(x_1,x_2)$. 
Let $\smatp{i}=\rho_i(\fs)$ for $i=1,2$ and $s =\rho(\fs)$. Then
for any $X,Y \in \irr(\CC)$, we have
$$
s_{X,Y} = \smatp{1}_{x_1,y_1}\smatp{2}_{x_2,y_2}U_{(x_1,x_2)}U_{(y_1,y_2)},
$$
where $\Phi(X)=(x_1, x_2), \Phi(Y)=(y_1, y_2) \in E_1 \times E_2$. By the Verlinde formula, for any  $X, Y \in D$ and $Z \in \irr(\CC)$,  we have
$$
\begin{aligned}
N_{X, Y}^{Z}
&=
\sum_{W \in \irr(\CC)}
\frac{s_{X,W} s_{Y,W} \ol{s_{Z,W}}}
{s_{\1,W}}\\
&=
\sum_{(w_1,w_2)\in B}
\frac{
\smatp{1}_{x_1,w_1}\smatp{1}_{y_1,w_1}\ol{s^{(1)}_{z_1,w_1}}
\left(\smatp{2}_{b_2,w_2}\right)^2\ol{s^{(2)}_{z_2,w_2}}
U_{(x_1,b_2)}U_{(y_1,b_2)}U_{(z_1,z_2)}U_{(w_1,w_2)}^3}
{
\smatp{1}_{b_1,w_1}\smatp{2}_{b_2,w_2}U_{(b_1,b_2)}U_{(w_1,w_2)}}\\
\end{aligned}
$$
where $\Phi(X)= (x_1, b_2),\ \Phi(Y)= (y_1, b_2),\ \Phi(Z)= (z_1, z_2)$ and $\Phi(W)= (w_1, w_2)$. Since $U_{(w_1,w_2)}^2 = 1$, we have
$$
\begin{aligned}
N_{X, Y}^{Z}
&=
\frac{U_{(x_1,b_2)}U_{(y_1,b_2)}U_{(z_1,z_2)}}{U_{(b_1,b_2)}}
\sum_{w_1 \in B_1}
\frac{\smatp{1}_{x_1,w_1}\smatp{1}_{y_1,w_1}\ol{s^{(1)}_{z_1,w_1}}}
{\smatp{1}_{b_1,w_1}}
\sum_{w_2 \in B_2}
\smatp{2}_{b_2, w_2}\ol{s^{(2)}_{z_2, w_2}}\\
&=
\delta_{b_2, z_2}
\frac{U_{(x_1,b_2)}U_{(y_1,b_2)}U_{(z_1,z_2)}}{U_{(b_1,b_2)}}
\sum_{w_1 \in B_1}
\frac{\smatp{1}_{x_1,w_1}\smatp{1}_{y_1,w_1}\ol{s^{(1)}_{z_1,w_1}}}
{\smatp{1}_{b_1,w_1}}\,,
\end{aligned}
$$
where the last equality is 
based on 
the fact that $\smatp{2}$ is symmetric and unitary.  Therefore, $N_{X,Y}^Z=0$ whenever $Z \not\in D$. Thus, $\DD$ is closed under the tensor product of $\CC$ and hence a fusion subcategory.

By Theorem~\ref{cor:fus-subcat}, $\DD$ is a modular subcategory of $\CC$. Since $p > 3$, we have $|\irr(\DD)| = |E_1| = \varphi_2(p) > 1$ so $\DD$ is nontrivial. Moreover, since $\CC$ is prime,  $\CC=\DD$ and so $\varphi_2(p)=|\irr(\CC)|$. Therefore, $\varphi_2(q)=1$, a contradiction! Therefore, $N$ is a prime.
\end{proof}

Now, we can prove our major theorem of this section.

\begin{thm}\label{thm:prime-transitive}
Let $\CC$ be a nontrivial transitive prime modular category. Then $\CC$ is equivalent to $\BB_{p-2, l}$ for some prime $p > 3$ and $l \in \UZ{2p}$ as modular categories. Moreover, the set 
$$
\{\BB_{p-2, l} \mid l \in \UZ{2p}\}
$$
is a complete set of inequivalent transitive prime modular categories whose T-matrices are of order $p$.
\end{thm}
\begin{proof}
Suppose $\CC$ is a nontrivial transitive prime modular category, then by Theorem \ref{thm:ordT-prime}, $ \ord(T_{\CC})$ is a prime $p > 3$. It follows from \cite[Lem.~2.2]{DongLinNg} that there exists a level $p$ representation $(\rho, V_\CC)$ of $\SLZ$ associated with $\CC$. Let $(s,t)$ denote the corresponding normalized modular data $(s,t)$ of $\CC$.  By Theorem \ref{thm:tran_irr},  $(\rho, V_\CC)$  is minimal and irreducible. In view of Lemma  \ref{l:realizing_min}, there exists a modular category $\DD=\BB_{p-2, l}$ for some $l \in \UZ{2p}$ and a level $p$ representation $(\rho', V_\DD)$ associated with $\DD$ such that
$$
(\rho, V_\CC)\cong (\rho', V_\DD)\,.
$$
Let $(s',t')$ be the normalized modular data of $\DD$ corresponding to $(\rho', V_\DD)$. Recall that $\irr(\DD)=\{V_d \mid d \in D\}$ where $D=\{2j\mid 0 \le j \le (p-3)/2\}$. We simply write $e_a$ for the basis element $e_{V_a}$ for $V_\DD$, and the entry $\rho'(\fa)_{V_a, V_b}$ as $\rho'(\fa)_{a,b}$ for any $\fa \in \SLZ$. As in the proof of Theorem \ref{thm:ordT-prime}, we have a bijection  $\Phi: D \to \irr(\CC)$ by comparing the eigenvalues of the images of $\ft$: for $a \in D$, we define 
$\Phi(a) := X \in \irr(\CC)$ if $\rho(\ft)_{X, X} = \rho'(\ft)_{a,a}$.

To simplify notations, we denote $s_{\Phi(a), \Phi(b)}$ by $s_{a,b}$ for any $a,b \in D$. By \cite[Lem.~3.17]{BNRWClassificationByRank}, there exists a diagonal matrix $U$, indexed by $D$,  of order at most 2  such that 
 $$
 s = U s' U\,.
$$
Let $x=\Phi^{-1}(\1)$. Then for any $a, b, c \in D$, the Verlinde formula yields the equations
$$
N_{\Phi(a), \Phi(b)}^{\Phi(c)}
=
\sum_{j \in D}
\frac{s_{a, j}s_{b, j}\ol{s_{c, j}}}{s_{x, j}}
=
\frac{U_{a,a}U_{b,b}U_{c,c}}{U_{x,x}}
\sum_{j \in D}
\frac{s'_{a, j}s'_{b, j}\ol{s'_{c, j}}}{s'_{x, j}}\,.
$$
 Since $\DD$ is transitive, there exists $\s \in \gal(\BQ_p/\BQ)$ such that $\hs(V_0) =V_x$ and we simply write $\hs(0)=x$. Applying $\s$ to the preceding equation, we find
 $$
\begin{aligned}
N_{\Phi(a), \Phi(b)}^{\Phi(c)}
&=
\frac{U_{a,a}U_{b,b}U_{c,c}}{U_{x,x}}
\sum_{j \in D}
\frac{\e_\s(a)\e_\s(b)\e_\s(c)}{ \e_\s(x)} \frac{s'_{\hs(a), j}s'_{\hs(b), j}\ol{s'_{\hs(c), j}}}{s'_{0, j}}\\
&=\pm N_{V_{\hs(a)}, V_{\hs(b)}}^{V_{\hs(c)}}\,.
\end{aligned}
$$
Since the fusion coefficients  $N_{\Phi(a), \Phi(b)}^{\Phi(c)}$ and $N_{V_{\hs(a)}, V_{\hs(b)}}^{V_{\hs(c)}} $ are nonnegative,  we have $N_{\Phi(a), \Phi(b)}^{\Phi(c)} = N_{V_{\hs(a)}, V_{\hs(b)}}^{V_{\hs(c)}} $ for all $a, b, c \in D$. Therefore, the assignment 
$$\Phi(a) \mapsto V_{\hs(a)}, \quad \text{for }a \in D,$$
defines a $\BZ_+$-based ring isomorphism between $K_0(\CC)$ and $K_0(\DD)$. By Lemma \ref{lem:FK-so3}, $\CC$ is equivalent to $\BB_{p-2,l}$ as modular categories for some $l \in \UZ{2p}$.

The second statement is an immediate consequence of Lemma \ref{lem:FK-so3} and Proposition \ref{p:B-trans-prime}.
\end{proof}

Finally, we establish the complete classification of nontrivial transitive modular categories.

\begin{thm}\label{thm:main}
Let $\CC$ be a nontrivial modular category. Then $\CC$ is transitive if and only if $\CC$ is equivalent to a Deligne product $\boxtimes_{a = 1}^{\ell}\BB_{p_{a}-2, l_{a}}$ as modular categories for some distinct primes $p_1, \dots, p_\ell > 3$ and $l_a \in \UZ{2p_a}$ . 
\end{thm}
\begin{proof}
If $\CC$ is transitive, then by Theorem \ref{thm:tran_irr}, $\ord(T_{\CC})=p_1\cdots p_\ell$ for some distinct primes $p_1, \dots, p_\ell > 3$. It follows from Theorem \ref{thm:prime-decomp}, $\CC$ can be uniquely factorized into a Deligne product of transitive prime modular categories up to the ordering of the factors. Therefore, by Theorem \ref{thm:prime-transitive}, $\CC$ is equivalent to $\boxtimes_{a = 1}^{\ell}\BB_{p_{a}-2, l_{a}}$ as modular categories for some $l_a \in \UZ{2p_a}$.

The converse of the statement follows directly from Proposition \ref{p:prod_A}.
\end{proof}

In view of Theorem \ref{thm:main}, nontrivial transitive modular categories $\CC$ up to equivalence are uniquely parameterized by a pair $(n,l)$ in which $n=\ord(T_\CC)$ is a square-free integer relatively prime to $6$ and $l$ is a congruence class in $\UZ{2n}$, which can be determined by the anomaly $\a_1(\CC)$.

\section{Transitivity of super-modular categories}\label{sec:super}
In this section, we investigate super-modular categories with transitive Galois actions. We first recall the definition of super-modular categories and the Galois group actions on their \emph{reduced} S-matrices.

The tensor category of $\Zn{2}$-graded finite-dimensional vector spaces over $\BC$ equipped with the super braiding is denoted by $\svec$. This braided fusion category $\svec$ is  symmetric and it can be endowed with two inequivalent spherical structures. The nontrivial simple object $f \in \svec$ is a \emph{fermion} that means $f \o f \cong \1$ and $\beta_{f,f} = -\id_{f\otimes f}$. The two inequivalent spherical structures on $\svec$ are distinguished by $d_f = \pm 1$. The corresponding premodular categories are respectively denoted by $\svec_\e$ with $d_f = \e$.

A premodular category $\CC$ is called \emph{super-modular} or \emph{super-modular category over $\svec_\e$} if $\CC'$ is equivalent to $\svec_\e$ as premodular categories for some $\e=\pm 1$. Let $f$ be the transparent fermion of $\CC$. Then, for any $X \in \irr(\CC)$, we have $d_{X\otimes f} = \e d_X$ and $\theta_{X\otimes f} = -\e\theta_X$ by the twist equation \eqref{eq:tw-def}.  Hence, $f \o X \not\cong X$. The transparent fermion $f \in \CC$ may also be denoted by $f_\CC$ if the context needs to be clarified.

The group $\irr(\CC') = \{\1, f\} \cong \Zn{2}$ acts on $\irr(\CC)$ by tensor product. We denote by $\ol X$  the $\Zn{2}$-orbit $\{X, f\o X\}$ of $\irr(\CC)$, and set of $\Zn{2}$-orbits  by $\ol{\irr(\CC)}$. By the above discussions, this $\Zn{2}$-action is fixed-point free, and so there exists a complete set of representatives $\Pi_\CC$ of $\ol{\irr(\CC)}$ such that $\1 \in \Pi_\CC$ and $\Pi_\CC$ is closed under taking duals. We call such a set $\Pi_\CC$ of simple objects of $\CC$ 
a \emph{basic subset of $\irr(\CC)$}, and 
we 
simply denote $\Pi_\CC$ by $\Pi$ when there is no ambiguity.  In general, $\irr(\CC) = \Pi \cup (f\otimes \Pi)$ and there is no canonical choice of $\Pi$ unless $\CC$ is a \emph{split} super-modular category, i.e., $\CC \simeq \DD\boxtimes\svec_\e$ as premodular categories for some modular category $\DD$ and some $\e \in\{\pm 1\}$. We call a super-modular category $\CC$ \emph{non-split} if $\CC$ is not a  split super-modular category.

With respect to the decomposition $\irr(\CC) = \Pi \cup (f\otimes \Pi)$, the  S-matrix of $\CC$ admits the block form
$$
S = 
\begin{pmatrix}
\hat{S} & d_f\hat{S}\\
d_f\hat{S} & \hat{S}
\end{pmatrix}\,,
$$
where $\hat{S}$ is a symmetric invertible matrix indexed by $\Pi$, called the \emph{reduced $S$-matrix} of $\CC$. The reduced S-matrix $\hat{S}$ of $\CC$ has the unitary normalization $\hat{s}=\frac{\sqrt{2}}{\sqrt{\dim(\CC)}} \hat{S}$ which satisfies a Verlinde-like formula 
\cite{NRWZ20a}. 
Since $\CC$ embeds into $Z(\CC)$ as a premodular 
subcategory, 
$S$ is defined over $\BQ_N$ where $N$ is the Frobenius-Schur exponent of $\CC$ or the order of the T-matrix of $Z(\CC)$ (cf. \cite{NS10}). The reduced S-matrix of $\CC$ will be denoted by $\hat{S}_\CC$ when the clarification is necessary.

It is immediate to see that $\BQ(\hat{S}) = \BQ(S) \subseteq \BQ_N$.  Similar to modular categories, we define  $G_\CC:=\gal(\BQ(S)/\BQ)$.  By \cite[Sec.\ 2.2]{NRWZ20a}, for any extension $E$ over $\BQ(S)$ and $\s \in \gal(E/\BQ)$, there exists a unique permutation $\hs$ on $\Pi$ satisfying 
\begin{equation}\label{eq:shat-action}
\s\left(
\frac{\hat{S}_{X, Y}}{\hat{S}_{\1, Y}}
\right)
=
\frac{\hat{S}_{X, \hs(Y)}}{\hat{S}_{\1, \hs(Y)}}
\end{equation}
for any $X, Y \in \Pi$ (see also \cite{BPRZ19a}). The permutation $\hs$ on $\Pi$ induces a permutation on $\ol{\irr(\CC)}$, namely $\hs(\ol X) := \ol {\hs(X)}$ for $X \in \Pi$, and we denote this permutation on  $\ol{\irr(\CC)}$ by the same notation $\hs$. This gives rise to an action of $\gal(E/\BQ)$ on $\ol{\irr(\CC)}$ by the restriction to $\BQ(S)$. In particular, $\GQ$ acts on $\ol{\irr(\CC)}$. Note that the action of $\GQ$ on $\ol{\irr(\CC)}$ is independent of the choices of $\Pi$.

Since the group homomorphism $\ \hat{\cdot} : G_\CC \to \Sym(\ol{\irr(\CC)})$ injective, we will identify $G_\CC$ with the image of $\ \hat\cdot\ $ as for modular categories. In other words, for any $\s \in \GQ$, we use $\hs$ to denote both the Galois automorphism on $\BQ(S)$ and the associated permutation on $\ol{\irr(\CC)}$. Again, we denote the $\GQ$-orbit of $\ol{\irr(\CC)}$  by $\Orb(\CC)$.

\begin{defn}
We call a super-modular category $\CC$ \emph{transitive} if the $\GQ$-action on $\ol{\irr(\CC)}$ is transitive.    
\end{defn}

We first derive some properties of the Galois actions on super-modular categories. The following lemma is an analog of \cite[Prop.~3.6]{RankFinite}.

\begin{lem}\label{lem:unit-dim}
Let $\CC$ be a super-modular category. Then for any $\s \in \GQ$, $d_{X}$ is a totally real algebraic unit for  $X \in \hs(\ol \1)$.
\end{lem}
\begin{proof}
Let $\Pi$ be a basic subset of $\irr(\CC)$. By \cite[Lem.~2.2]{NRWZ20a}, for any $\s \in \GQ$, we have 
\begin{equation}\label{eq:2}
d_{\hs(\1)}^2 = \frac{\dim(\CC)}{\s(\dim(\CC))}\,. 
\end{equation}
Since $\frac{\dim(\CC)}{\s(\dim(\CC))}$ has algebraic norm 1, and  $d_{\hs(\1)}$ is a totally real algebraic integer (see \cite{ENO}),  $d_{\hs(\1)}$ is a a totally real algebraic unit.  Now the statement follows from the fact  that $\hs(\ol{\1}) = \{\hs(\1), f \o \hs(\1)\}$ and $d^2_{f \o \hs(\1)} = d^2_{\hs(\1)}$. 
\end{proof}

On split transitive super-modular categories, we begin with the following lemma.

\begin{lem}\label{l:split_super}
Let $\DD$ be a modular category. Then the split super-modular category $\CC = \DD\boxtimes\svec_\e$ for any $\e=\pm 1$ is transitive if and only if $\DD$ is transitive.
\end{lem}
\begin{proof}
We can take $\Pi = \irr(\DD)$. The $\GQ$-action on $\ol{\irr(\CC)}$ is equivalent to its action on $\Pi$, which coincides with the $\GQ$-action on the modular category $\DD$. Therefore, the statement follows.
\end{proof}

Combining Lemma \ref{l:split_super} and Theorem \ref{thm:main}, we obtain the full classification of split transitive  super-modular categories. 

\begin{thm}\label{thm:split}
Let $\CC$ be a nontrivial split super-modular category. Then $\CC$ is transitive if and only if $\CC$ is equivalent to $\left(\boxtimes_{a = 1}^{\ell}\BB_{p_{a}-2, l_{a}}\right) \boxtimes \svec_\e$  as premodular categories for some $\e \in \{\pm 1\}$, distinct primes $p_1, \dots, p_\ell > 3$ and $(l_1, \dots, l_\ell) \in \UZ{2p_1}\times \cdots\times\UZ{2p_\ell} $. \qed
\end{thm}

Transitive super-modular categories have similar properties as transitive modular categories. For example, the following lemma is parallel to Proposition \ref{p:sizeofGalois}.

\begin{lem}
If $\CC$ is a transitive super-modular category, then we have $|G_\CC| = |\irr(\CC)|/2$.
\end{lem}
\begin{proof}
Since $G_\CC$ acts transitively on $\ol{\irr(\CC)}$, $G_\CC$ is regular and so 
\begin{equation*}
|G_\CC| = |\ol{\irr(\CC)}|= |\irr(\CC)|/2. 
\qedhere
\end{equation*}
\end{proof}

Therefore, for any transitive super-modular category $\CC$ with a basic subset $\Pi$  of $\irr(\CC)$, we can identify $G_{\CC}$ with $\Pi$ via $\hs \mapsto \hs(\1)$. Under this identification,  we will simply denote $f \otimes\hs$ by $f\hs$ for any $\hs \in G_{\CC}$. Now, we can compare the following theorem to Lemma \ref{lem:1} and Theorem \ref{thm:1}.

\begin{thm}\label{thm:super}
Let $\CC$ be a transitive super-modular category with a basic subset $\Pi$  of $\irr(\CC)$. Let $\shat$ be the reduced S-matrix of $\CC$ indexed by $G_\CC$ according to the preceding identification of $G_\CC$ and $\Pi$. Then:
\begin{itemize}
\item[(i)] For any $\hs, \hm \in G_{\CC}$, we have $\shat_{\hs, \hm} = \hs(d_{\hm}) d_{\hs} = \hm(d_{\hs})d_{\hm}$. In particular, all entries of $\shat$ and $S$ are totally real algebraic units. 
\item[(ii)] For any $\hs, \hm \in G_{\CC}$, if $\hs \ne \hm$, then $d_{\hm}^2 \ne d_{\hm}^2$. In particular, if $\hm \ne \1$, then $d_{\hm}^2 \ne 1$ and $\hm(\dim(\CC)) \ne \dim(\CC)$. 
\item[(iii)] For any $X \in \irr(\CC)$, if $d_X^2 \in \BZ$, then $ X \in \{\1, f\}$. For any fusion subcategory $\DD \subset \CC$, if $f \in \DD$, then $\DD$ is a super-modular category, otherwise, $\DD$ is a modular category. In particular, $\CC$ has no nontrivial Tannakian subcategory.
\item[(iv)] $\BQ(\shat) = \BQ(\dim(\CC)) = \BQ(d_\hs \mid \hs \in G_{\CC})$.
\end{itemize}
\end{thm}
\begin{proof}
The first equality of statement (i) follows from \eqref{eq:shat-action} by setting $X = \hm$, $Y = \1$, and the second equality follows from the fact that $\shat$ is symmetric. Consequently, by Lemma \ref{lem:unit-dim}, all entries of $\shat$ and $S$ are totally real algebraic  units.

Now we have (i) and \eqref{eq:2}, the proof of (ii) and (iv) are the same as that of Theorem \ref{thm:1} (i) and (iii) by replacing $S$ by $\shat$. 

For statement (iii), assume $X \in \irr(\CC)$ satisfies $d_{X}^2 \in \BZ$. By the above discussions, since $\CC$ is transitive, there exists $\hs\in G_{\CC}$ such that $X = \hs$ or $X = f\hs$. In either case, we have $d_{X}^2 = d_{\hs}^2 \in \BZ$. By Lemma \ref{lem:unit-dim}, $d_{\hs}$ is a real algebraic unit, so $d_X^2 = d_{\hs}^2 = 1$. Consequently, by (ii), we have $\hs = \1$. Therefore, $X = \1$ or $X = f$. 

Let $\DD\subset\CC$ be any fusion subcategory. Then  $\DD$ is a premodular subcategory of $\CC$. Since the M\"uger center $\DD'$ of $\DD$ is a symmetric fusion subcategory of $\CC$, we have $d_X^2 \in \BZ$ for any $X \in \DD'$. Therefore, we have $\irr(\DD') \subset \{\1, f\}$ and hence  $\DD$ is super-modular (resp.~modular) if and only if $f \in \DD$ (resp.~$f \not\in \DD$). Finally, if  $\DD$ is a  Tannakian subcategory of $\CC$, then $f \notin \DD$ and so $\DD$ is modular. Therefore, $\DD$ braided equivalent to $\Vec{}$, and this completes the proof of the theorem.
\end{proof}

By definition, a super-modular category over $\svec_\e$ for some $\e=\pm 1$ is a nondegenerate braided fusion category over $\svec$ according to \cite{DNO}. Therefore, if  $\AA$ and $\B$ are super-modular categories over $\svec_\e$, then their tensor product $\AA \btsvec \B = (\AA \bot \B)_A$ is a nondegenerate braided fusion category over $\svec$, where $A = \1_\AA \bot \1_\B \oplus f_\AA \bot f_\B$ is a connected \'etale algebra in $\AA \bot\B$. It is immediate to see that $\dim(A)=2$ and $\theta_A = \id_A$ for any $\e = \pm 1$. Therefore, $\AA \btsvec \B$ admits a spherical structure inherited from $\AA \bot \B$ by \cite{KO02}
, which implies that 
$\AA \btsvec \B$ is a super-modular category.

Consider the forgetful functor $G: \AA \btsvec \B \to \AA \bot \B$, and the free-module functor $F:\AA \bot \B \to \AA \btsvec \B$ defined by $F(X \bot Y) = (X \bot Y) \o A$ for $X \in \AA, Y \in \B$. According to \cite{KO02}, $F$ is a surjective tensor functor, and $G$ is right adjoint to $F$. Let $\dim_A(M)$ denote the categorical dimension of any object $M \in \AA \btsvec \B$. We have
\begin{equation*}
\dim_A(F(X \bot Y)) = \dim_{\AA \bot \B}(X \bot Y) = \dim_\AA(X)\cdot \dim_\B(Y)
\end{equation*}
for any $X \in \AA$ and $Y \in \B$.
Therefore, $F: \AA \bot \B \to \AA \btsvec \B$ preserves the spherical structures (cf.  \cite{NSind}).

Since $f_\AA \bot f_\B$ acts freely on $\irr(\AA \bot \B)$, $F(X \bot Y)$ is simple  
for any $X \in \irr(\AA)$ and $Y \in \irr(\B)$. The transparent fermion of $\AA \btsvec \B$ is given by
\begin{equation} \label{eq:F_fermion}
F(f_\AA \bot \1_\B) \cong f_\AA \bot \1_\B \oplus \1_\AA \bot f_\B \cong F(\1_\AA \bot f_\B)
\end{equation}
and 
$$
\dim_A(F(f_\AA \bot \1_\B)) = \dim_\AA(f_\AA) = \e\,.
$$
Therefore, $\AA \btsvec \B$ is a super-modular category over $\svec_\e$. This proves the first statement of the following lemma.

\begin{lem}\label{l:adm}
Let $\AA$ and $\B$ be super-modular categories over $\svec_\e$  for some $\e\in\{\pm 1\}$. Then: 
\begin{itemize}
\item[(i)] $\CC:=\AA \btsvec \B$ is a super-modular category over $\svec_\e$, 
$$
\irr(\CC)=\{F(X \bot Y)\mid (X, Y) \in\irr(\AA) \times \irr(\B)\}\,,
$$
and
$$
\dim_A(F(X \bot Y)) = d_X d_Y
$$
for any $X \in \AA$ and $Y \in \B$.
\item[(ii)] Let $\Pi_\AA$ and $\Pi_\B$ be basic subsets of  $\irr(\AA)$ and $\irr(\B)$ respectively. Then
$$
\Pi_{\CC} = \{F(X \bot Y)\mid (X,Y) \in \Pi_\AA \times \Pi_\B\}
$$
is a basic subset of $\irr(\CC)$. Moreover, the corresponding reduced S-matrix $\shat_\CC$ of $\CC$ is given by the Kronecker product $\shat_\CC=\shat_\AA \o \shat_\B$.
\end{itemize}
\end{lem}
\begin{proof} We continue the preceding discussions
to prove
(ii). For any $X \in \irr(\AA)$ and $Y \in \irr(\B)$, we have 
$$
GF(X \bot Y) \cong X \bot Y \oplus (X \o f_\AA)  \bot (Y\o f_\B)\,.
$$
Therefore, for any $(X, Y)\ne (X', Y') \in\irr(\AA) \times \irr(\B)$, 
$$
F(X \bot Y)\cong F(X' \bot Y')\ \text{ if and only if }\ X' \bot Y' \cong (X \o f_\AA)  \bot (Y\o f_\B)\,.
$$
For any $(X, Y), (X', Y') \in \Pi_\AA \times \Pi_\B$, $X' \bot Y' \not\cong (X \o f_\AA)\bot (Y\o f_\B)$ by the definition of a basic subset.  Since $F$ is a tensor functor, $\1_\CC \in \Pi_\CC$ and $\Pi_\CC$ is closed under taking dual. It follows from \eqref{eq:F_fermion} that
$$
\irr(\CC) = \Pi_\CC \cup (f_\CC \o \Pi_\CC)
$$
where $f_\CC = F(\1_\AA \bot f_\B)$\,. Therefore, $\Pi_\CC$ is a basic subset  of $\irr(\CC)$.

By \cite[Thm. 4.1]{KO02}, for any $(X,Y), (X',Y') \in \Pi_\AA \times \Pi_\B$,
\begin{eqnarray*}
\dim(A) (S_\CC)_{\ul{X \bot Y}, \ul{X' \bot Y'}} & = &  (S_{\AA \bot \B})_{X \bot Y, X' \bot Y'} + (S_{\AA \bot \B})_{X \bot Y, (f_\AA \o X') \bot (f_\B\o Y')} \\
& = & 2 (S_{\AA})_{X,X'} (S_\B)_{X',Y'} 
\end{eqnarray*}
where $\ul{X \bot Y} = F(X \bot Y)$. Since $\dim(A)=2$, we have $$(S_\CC)_{\ul{X \bot Y}, \ul{X' \bot Y'}} = (S_{\AA})_{X,X'} (S_\B)_{X',Y'},$$ 
which is equivalent to $\hat{S}_{\CC} = \hat{S}_{\AA} \otimes \hat{S}_{\B}$.
\end{proof}

Recall the definition of the fiber product in Section \ref{s:galois}. 

\begin{cor}\label{cor:sup-orb}
Let $\AA$, $\B$ be super-modular categories over $\svec_\e$ for some $\e = \pm 1$ with basic subsets $\Pi_\AA$ and $\Pi_\B$ of simple objects of $\AA$ and $\B$ respectively. Let  $\CC :=\AA\btsvec\B $  and $\BF = \BQ(S_\AA)\cap \BQ(S_\B)$.  Then:
\begin{itemize}
\item[(i)] 
The map $g: G_{\CC} \to G_{\AA}\bullet G_{\B}$, $g(\hs_{\CC}) = (\hs_{\AA}, \hs_{\B})$, defines an isomorphism of groups, and 
\begin{equation*}
|G_{\CC}| = \frac{|G_{\AA}| \cdot |G_{\B}|}{[\BF:\BQ]}\,.    
\end{equation*}
\item[(ii)] 
For any $\s \in \GQ$, $X \in \Pi_\AA$ and $Y \in \Pi_\B$, we have
$$
\hs_{\CC}(F(X \bot Y)) = F(\hs_\AA(X)\bot \hs_{\B}(Y))\,.
$$
\item[(iii)] 
$|\Orb(\AA)|\cdot|\Orb(\B)| \le |\Orb(\CC)| \le |\Orb(\AA)|\cdot|\Orb(\B)|\cdot [\BF:\BQ]$.
\item [(iv)]
If $\AA$ and $\B$ are transitive, then 
$$
|\Orb(\CC)| 
= 
\frac{|G_{\AA}|\cdot |G_{\B}|}{|G_{\CC}|}
=
[\BQ(\dim(\AA)) \cap \BQ(\dim(\B)): \BQ]\,.
$$
\end{itemize}
\end{cor}
\begin{proof}
In view of Lemma \ref{l:adm}, by replacing $\irr(\AA), \irr(\B), \irr(\CC)$ respectively with $\Pi_\AA, \Pi_\B$, and the associated $\Pi_\CC$, the statements (i)-(iii) can be proved in the same way as Lemma \ref{lem:gal-on-bot}, and the proof of (iv) is similar to that of Proposition \ref{prop:trans-prod}.
\end{proof}

\begin{prop}\label{p:sph}
Let $\CC$ be a super-modular category over $\svec_\e$ for some $\e=\pm 1$.  If $\AA$ is a super-modular subcategory of $\CC$, then  both $\AA$ and its M\"uger centralizer $\B = C_{\CC}(\AA)$ are super-modular categories over $\svec_\e$, and there is an equivalence of premodular categories over $\svec$, 
\begin{equation}\label{eq:centralizer}
    \CC \simeq \AA \btsvec \B \,.
\end{equation}
\end{prop}
\begin{proof} It is clear that $\AA$ is a super-modular over $\svec_\e$. Note that $\CC'$ is a premodular subcategory of $\B$, which is a nondegenerate braided fusion category over $\svec$ by \cite[Prop.~4.3]{DNO}. Therefore, by Lemma \ref{l:adm}, $\B$ and $\AA \btsvec \B$ are super-modular categories over $\svec_\e$. 

By \cite[Prop.~4.3]{DNO}, there exists a braided tensor equivalence $\AA \btsvec \B \simeq \CC$ over $\svec$. In fact, the tensor product functor $\o : \AA \boxtimes\B \to \CC$, $X \bot Y \mapsto X \o Y$, for any $X \in \AA$ and $Y \in \B$, defines an essentially surjective braided tensor functor. This braided tensor functor descends to a braided tensor equivalence $\ol{\otimes}: \AA\btsvec\B \xrightarrow{\sim} \CC$ over $\svec$, which satisfies the commutative diagram
\begin{equation}
\begin{tikzcd}
\AA\boxtimes \B \ar[rr, "\otimes"]\ar[dr, "F"'] && \CC\\
& \AA\btsvec\B \ar[ur, "\ol{\otimes}"']
\end{tikzcd}\,.
\end{equation}
By Lemma \ref{l:adm}, any simple object in $\AA\btsvec\B$ is isomorphic to $F(X\boxtimes Y)$ for some $(X, Y) \in \irr(\AA)\times \irr(\B)$ and 
$$
\dim_A(F(X\boxtimes Y)) = d_Xd_Y = d_{X\otimes Y} = d_{\ol{\otimes}(F(X\boxtimes Y))}\,.
$$
Therefore, $\ol{\otimes}$ preserves spherical structures, and hence is an equivalence of premodular categories.
\end{proof}

\begin{cor}\label{cor:sup-trans-sub}
Let $\CC$ be a transitive super-modular category. Then any fusion subcategory of $\CC$ is  transitive modular or super-modular.
\end{cor}
\begin{proof}
By Theorem \ref{thm:super} (iii), any fusion subcategory $\AA\subset \CC$ is either modular or super-modular. Assume first that $\AA$ is super-modular. In view of Proposition \ref{p:sph} and Corollary \ref{cor:sup-orb}, the proof of transitivity of $\AA$ is the same as that of Theorem \ref{thm:prime-decomp} with the sets $\irr(\AA), \irr(\B)$ and $\irr(\CC)$ of irreducible objects replaced by basic sets of simple objects $\Pi_\AA$, $\Pi_\B$ and the corresponding $\Pi_\CC$. Now, we assume $\AA$ is modular. Then $\DD := \AA\vee\CC'$, the fusion subcategory of $\CC$ generated by $\AA$ and $\CC'$, is a super-modular subcategory of $\CC$. By the above discussions, $\DD$ is transitive. Therefore, by Lemma \ref{l:split_super}, $\AA$ is transitive. 
\end{proof}
\begin{cor}\label{cor:sup-trans2}
Let $\AA$, $\B$ be super-modular over $\svec_{\e}$ for some $\e = \pm 1$. Then $\AA\btsvec\B$ is transitive if and only if the following two conditions hold: both $\AA$, $\B$ are transitive, and $\BQ(\dim(\AA))\cap \BQ(\dim(\B)) = \BQ$. 
\end{cor}
\begin{proof}
 Let $\CC = \AA\btsvec \B$ be transitive. Then both $\AA$ and $\B$ are transitive by Corollary \ref{cor:sup-trans-sub}. Therefore, by Corollary \ref{cor:sup-orb} (iv), we have 
 $$
 |\Orb(\CC)| = [\BQ(\dim(\AA)) \cap \BQ(\dim(\B)): \BQ] = 1,
 $$
 and so $\BQ(\dim(\AA)) \cap \BQ(\dim(\B)) = \BQ$.

Conversely, it follows immediately from Corollary \ref{cor:sup-orb} (iv) that if $\AA$ and $\B$ are transitive, and $\BQ(\dim(\AA)) \cap \BQ(\dim(\B)) = \BQ$, then $\CC$ is transitive.
\end{proof}

The following definition generalizes the primality of modular categories. 

\begin{defn}
Let $\EE$ be a symmetric fusion category, and $\CC$ a nondegenerate braided fusion category $\CC$  over $\EE$. We say that $\CC$ is \emph{$\EE$-prime} if it has no nondegenerate braided fusion subcategory over $\EE$ except $\EE$ and $\CC$. An $\EE$-prime braided fusion category is called \emph{$\EE$-simple} if it is not pointed. For $\EE=\svec$, we simply use the terms s-prime and s-simple instead of $\svec$-prime and $\svec$-simple.
\end{defn}

Note the definition of $\EE$-simple categories is consistent with the definition  of s-simple categories introduced in \cite{DNO}. We will call a super-modular category \emph{trivial} if it is braided equivalent to $\svec$. In particular, $\svec_{\pm 1}$ are trivial. In view of Theorem \ref{thm:super} (iii), nontrivial s-prime transitive super-modular categories are s-simple. Now we can state and prove the prime decomposition theorem for transitive super-modular categories (cf.~Theorem \ref{thm:prime-decomp}).

\begin{thm}\label{thm:sup-decomp}
Let $\CC$ be a nontrivial transitive super-modular category over $\svec_\e$ for some $\e=\pm 1$. 
Then 
\begin{equation}\label{eq:s-simple}
\CC \simeq \CC_1 \btsvec \cdots \btsvec \CC_m\,,
\end{equation}
as premodular categories, where $\CC_1, \dots, \CC_m$ form the complete list of inequivalent $s$-simple subcategories of $\CC$.
Moreover, such factorization into  s-simple super-modular categories over $\svec_\e$ of $\CC$ is unique up to permutation of factors.
\end{thm}
\begin{proof} By Theorem \ref{thm:super} (iii), $\CC$ has no Tannakian subcategory other than $\Vec$ and $\CC_\pt = \CC' \simeq \svec_\e$. According to  \cite[Thm. 4.13]{DNO} (i) and Proposition \ref{p:sph}
$$
\CC \simeq \CC_1 \btsvec \cdots \btsvec \CC_m
$$
as premodular categories for some s-simple subcategories $\CC_1, \dots, \CC_m$ of $\CC$. It follows from   Corollary \ref{cor:sup-trans2} that $\CC_1, ..., \CC_m$  are  transitive and 
$$
\BQ(\dim(\CC_i)) \cap  \BQ(\dim(\CC)/\dim(\CC_i)) = \BQ
$$
for any $i = 1, \dots, m$. In particular, these s-simple super-modular subcategories of $\CC$ have distinct global dimensions. According to \cite[Thm.~4.13]{DNO} (ii), $\CC_1, \cdots, \CC_m$ are all the s-simple super-modular subcategories of $\CC$. Thus, if 
$$
\CC \simeq \DD_1 \btsvec \cdots \btsvec \DD_n
$$
as premodular categories for some s-simple super-modular categories  $\DD_1, \dots, \DD_n$ over $\svec_\e$, then they are equivalent to a complete list of inequivalent s-simple  super-modular subcategories of $\CC$. Therefore, $m=n$ and the statement follows.
\end{proof}

Now, we demonstrate a family of  transitive  non-split super-modular categories derived from quantum group modular categories. 

According to \cite{Fermion17}, for any $k \ge 0$ and $l \in (\Zn{8(k+1)})^\times$, the category $\CC = \BB_{4k+2, l}$ (see Section \ref{sec:example}) is  super-modular 
with $\irr(\CC) = \{V_{2j}\mid 0 \le j \le 2k+1\}$. The fermion of $\CC$ is $V_{4k+2}$. By the fusion rules \eqref{eq:su2}, we have $V_{2j} \otimes V_{4k+2} = V_{4k+2-2j}$. In the following discussions, we choose
$$\Pi_0 = \{V_{2j} \mid 0 \le j \le k\}\,.$$
When $k = 0$, $\CC$ is braided equivalent to $\svec$, and when $k \ge 1$, $\CC$ is non-split.

\begin{prop}\label{p:prime_super}
For any $k \ge 1$, the super-modular category $\BB_{4k+2,l}$ is s-simple. 
\end{prop}
\begin{proof}
First, we show that any nontrivial fusion subcategory of $\CC$ is either $\CC$ or $\CC'$. 

Recall that $\CC_\pt=\CC'$  and  $\irr(\CC') = \{\1, V_{4k+2}\}$. Assume that $\DD$ is a   nontrivial fusion subcategory of $\CC$ and $\DD$ is not pointed. Then $\DD$ has a simple object $X$ which is not invertible, and so $X \cong $ $V_{2j}$ for some $1 \le j \le 2k$. In particular, we have $4j \ge 4$, and $2(4k+2)-4j  \ge 4$. So by the fusion rules, $N_{2j,2j}^{2} = 1$, which means $\DD$ contains $V_2$. Since $V_2$ tensor generates $\CC$, we have $\DD = \CC$. Therefore, $\CC$ is s-prime. Since $k \ge 1$, $\CC \ne \svec$, so it is s-simple.
\end{proof}

\begin{prop}\label{p:transitive_super}
Let $\CC=\BB_{4k+2,l}$ for some integer $k \ge 1$ and $l \in \UZ{8(k+1)}$. Then $\CC$ is transitive if and only if $k=2^x-1$ for $x \ge 1$. 
\end{prop}
\begin{proof}

Recall that the quantum parameter of $\CC$ is $q^l = \exp(\frac{l\pi i}{4(k+1)})$, and $\BQ(S)$ is a real subfield of $\BQ_{8(k+1)}$, so $|G_{\CC}|$ divides $\varphi(8(k+1))/2$, where $\varphi$ is the Euler phi function. Assume $\CC$ is transitive.  Then $|\Pi_0| = k+1$ must divide $\varphi(8(k+1))/2$. 

We first observe that $k$ must be odd. Suppose $k$ is even. Then  $k+1 \ge 3$ is an odd integer, and  so $\varphi(8(k+1))/2 = 2\varphi(k+1)$. Therefore,   $k+1 \mid \varphi(8(k+1))/2$ implies $k+1 \mid \varphi(k+1)$. This divisibility does not hold for any $k >0$. Therefore, $k$ must be odd. 

Let $k+1 = 2^xw$, where $x \ge 1$ and $w$ is odd. Then $\varphi(8(k+1))/2 = 2^{x+1}\varphi(w)$. Since $k+1$ divides $\varphi(8(k+1))/2$, we have $w \mid 2\varphi(w)$ and hence $w \mid \varphi(w)$. This can only happen when $w = 1$, or equivalently, $k=2^x-1$.

Conversely, assume $k = 2^x-1$ for $x \ge 1$ and $\CC = \BB_{4k+2,l}$. With respect to our choice of $\Pi_0$, for any $0 \le a, b \le k$, we have
$$
\hat{S}_{2a,2b} = [(2a+1)(2b+1)]_{q^{l}}\,.
$$
Following the same argument as in the proof of Proposition \ref{p:B-trans-prime},  one can show that $\CC$ is transitive.  More precisely, since $\BQ(S) \subset \BQ_{8(k+1)} = \BQ_{2^{x+3}}$, for any $0 \le j \le k$, we have $\gcd(2j+1, 2^{x+3}) = 1$. So there exists $\s \in \BQ_{2^{x+3}}$ such that $\s(q) = q^{2j+1}$. Therefore,
$$
\s\left(\frac{\hat{S}_{2i,0}}{\hat{S}_{0,0}}\right)
=
\s([2i+1]_{q^{l}})
=
\frac{[(2i+1)(2j+1)]_{q^{l}}}{[2j+1]_{q^{l}}}
=
\frac{\hat{S}_{2i,2j}}{\hat{S}_{0,2j}}\,.
$$
In other words, $\hs(V_{0}) = V_{2j}$, and hence $\CC$ is transitive.
\end{proof}

In light of Theorem \ref{thm:main}, it is natural to ask whether there are other non-split transitive super-modular categories that are s-simple, and we propose the following question at the end this paper.

\begin{conj}
The quantum group categories in Proposition \ref{p:transitive_super} are all the s-simple non-split transitive super-modular categories up to Galois conjugates and spherical structures.
\end{conj} 

\section*{Acknowledgements}
This paper is based upon work supported by the National Science Foundation under the Grant No.~DMS-1440140 while the first and the last authors were in residence at the Mathematical Sciences Research Institute in Berkeley, California, during the Spring 2020 semester. They would also like to thank Eric Rowell for fruitful discussions.

\bibliographystyle{abbrv}
\bibliography{zbib}

\end{document}